\documentclass[11pt,noindent]{article}
\usepackage{latexsym,amsmath,amsfonts}
\usepackage{color}
\usepackage{graphicx,amsmath,amssymb,epsfig}
\usepackage[shortcuts]{extdash}

\newtheorem{theorem}{Theorem}
\newtheorem{proposition}{Proposition}
\newtheorem{corollary}{Corollary}
\newtheorem{lemma}{Lemma}

\newtheorem{assumption}{Assumption}

\newtheorem{remark}{Remark}

\numberwithin{theorem}{section} \numberwithin{equation}{section}
\numberwithin{proposition}{section} \numberwithin{lemma}{section}
\numberwithin{corollary}{section}

\newcommand{\thmref}[1]{Theorem~\ref{thm:#1}} 
\newcommand{\lemref}[1]{Lemma~\ref{lem:#1}} 
\newcommand{\eqnref}[1]{(\ref{eq:#1})} 

\def\be{\begin{equation} }
\def\ee{ \end{equation}}

\def\ben{\begin{equation*}}
\def\een{\end{equation*}}
\def\bea{\begin{eqnarray}}
\def\eea{\end{eqnarray}}
\def\ee{\end{eqnarray}}
\def\bean{\begin{eqnarray*}}
\def\eean{\end{eqnarray*}}

\newcommand\ignore[1]{}

\def\R{\mathbb{R}} 
\def\Z{\mathbb{Z}} 
\def\N{\mathbb{N}} 

\newcommand{\Ex}[1]{\mathbb{E}\left[#1\right]} 

\newcommand{\Exp}[2]{\mathbb{E}_{#1}\left[#2\right]} 
\newcommand{\Ind}[1]{\chi_{#1}} 
\newcommand{\Var}[1]{\mathbb{V}\left(#1\right)} 
\renewcommand{\Pr}[1]{\mathbb{P}\left(#1\right)} 

\newcommand{\bigoh}[1]{O\left(#1\right)}

\def\sF{\mathcal{F}}
\def\sG{\mathcal{G}}

\def\sS{{\sf Supp}}

\newcommand\QED{\ifhmode\allowbreak\else\nobreak\fi
\quad\nobreak$\Box$\medbreak}
\newcommand{\proofstart}{\par\noindent\sl Proof:\rm\enspace}
\newcommand{\proofend}{\QED\par}
\newenvironment{proof}{\proofstart}{\proofend}

\def\eps{\epsilon}

\def\tr{{\rm Tr}}

\def\tr{{\rm Tr}}

\def\tr{{\rm tr}}

\def\Sum{{\sf \mathcal{S}}}
\newcommand{\Leaves}[1]{{#1}^\spadesuit}
\def\supp{{\rm supp}^* }

\begin{document}

\title{A high dimensional Central Limit Theorem for martingales, with applications to context tree models}
\author{Alexandre Belloni\thanks{The Fuqua School of Business, Duke University.} \, and Roberto I. Oliveira\thanks{IMPA, Rio de Janeiro, Brazil. }}

\maketitle

\abstract{We establish a central limit theorem for (a sequence of) multivariate martingales which dimension potentially grows with the length $n$ of the martingale. A consequence of the results are Gaussian couplings and a multiplier bootstrap for the maximum of a multivariate martingale whose dimensionality $d$ can be as large as $e^{n^c}$ for some $c>0$.  We also develop new anti-concentration bounds for the maximum component of a high-dimensional Gaussian vector, which we believe is of independent interest.

The results are applicable to a variety of settings. We fully develop its use to the estimation of context tree models (or variable length Markov chains) for discrete stationary time series. Specifically, we provide a bootstrap-based rule to tune several regularization parameters in a theoretically valid Lepski-type method. Such bootstrap-based approach accounts for the correlation structure and leads to potentially smaller penalty choices, which in turn improve the estimation of the transition probabilities.}

%

\section{Introduction}

Modern statistical applications often involve high dimensional settings. Typical problems include fitting complex models, tuning estimators with many parameters, and providing confidence regions for high-dimensional data. Prominent examples of methods include $\ell_1$-penalized estimators (e.g., Lasso and its variants), post-selection and debiased estimators, among many others. Addressing these challenges requires tools for quantifying uncertainty. Recent papers have addressed this need with Central Limit Theorems and bootstrap methods that work even when the problem dimension greatly exceeds the sample size \cite{chernozhukov2012comparison,chernozhukov2013gaussian,chernozhukov2014clt} for suitable statistics. These results have attracted substantial interest to handle modern applications (see \cite{belloni2018high} and the references therein) and have motivated many new theoretical developments \cite{chernozhukov2013testing,chernozhukov2015noncenteredprocesses,deng2017beyond,zhang2017gaussian,belloni2018subvector,lopes2018bootstrapping}. These contributions are discussed below.

In this paper, we contribute to this literature in two directions. Our first main (theoretical) contribution is a multidimensional martingale CLT and an attendant bootstrap method. Our second main contribution (on the applications side) is to show how the proposed bootstrap procedure can be used in the estimation of categorical time series via Variable Length Markov Chains (VLMC). Motivated by this application, we also obtain new anti\=/concentration results for the maximum component of a high-dimensional Gaussian vector that are of independent interest. 

\subsection{The martingale CLT and the bootstrap}

Martingales naturally arise in a variety of applications \cite{hall1980martingale}. They capture the key zero-mean increment property in many time series settings. In particular, Central Limit Theorems for martingales have attracted substantial interest. On the one hand, martingale CLTs imply results for processes with dependencies. On the other hand, they allow the construction of confidence intervals and hypothesis testing in substantially more general settings than the i.i.d. case.

In Section \ref{SEC:CLTs}, we establish a quantitative Central Limit Theorem for a (triangular sequence of) martingales whose dimension increases with its length (Theorem \ref{thm:lindeberg}). The dimensionality of the data and its interplay with functionals of interest plays a critical role in the analysis.

The CLT is specialized in Section \ref{SEC:MAX} to maximum-type functionals. Under some assumptions, we show Gaussian approximation and bootstrap results for the maximum component of high-dimensional martingales (Theorem \ref{thm:clt:max:martingale} and Corollary \ref{cor:bootstrap}). Most notably, our CLT gives good bounds in a truly high-dimensional setting when the martingale length $n$ and the dimension $d$ satisfy $\log^c d=o(n)$ for some positive $c>0$. This is illustrated by the simple application in Section \ref{SEC:MM}.

Our results build upon recent work on the maximum component of a sum of independent vectors \cite{chernozhukov2012comparison,chernozhukov2012gaussian,chernozhukov2013gaussian,deng2017beyond}, while at the same time incorporating dependencies. There has also been some related work on the dependent setting. Reference  \cite{chernozhukov2013testing} establishes the validity of a block multiplier bootstrap under $\beta$-mixing conditions. References \cite{zhang2014bootstrapping,zhang2017gaussian} consider functionally dependent time series data under different contraction conditions (see \cite{wu2005nonlinear}). In all these results the dimensionality of the vectors $d$ is potentially much larger than the sample size $n$. The proof strategy of these works that handle dependence relies on Slepian interpolation, which does not seem to generalize well to the martingale setting. We thus rely on an adaptation of the standard Lindeberg approach that works in the $d\gg n$ setting \footnote{\cite{deng2017beyond} uses a Lindeberg based proof but for the independent case.}.

Another aspect of our CLT for the maximum has to do with anti-concentration. Establishing a CLT requires that approximation errors vanish faster than the probabilities we are trying to estimate. Anti\=/concentration bounds for Gaussian vectors have been a key tool for establishing this property \cite{chernozhukov2012comparison,gotze2017gaussian}; however, these results require good control of the smallest variance of the Gaussian vector. This is a severe restriction in our VLMC application, where some variances may be very close to $0$ (if not $0$). To circumvent this problem, we give in Theorem \ref{thm:max:anticoncentration} an anti\=/concentration bound that depends only on the maximum variance. This is possible as long as we are concerned with the anti-concentration behavior in the tails, in contrast to the anti-concentration at all values that has been the focus so far. We believe our anti\=/concentration result will be useful in many other settings. 

We note that a very recent result in \cite{lopes2018bootstrapping} also derives anti\=/concentration of the maximum of a Gaussian vector with vanishing minimum variance. Their main assumption is that the variances of components have a suitable structure. Our result was developed independently and nearly simultaneously with \cite{lopes2018bootstrapping}.

\subsection{Categorical time series estimation via VLMC}

In Section \ref{SECCONTEXT} we use our CLT and bootstrap to improve an algorithm that estimates transition probabilities and dependency structures of discrete-alphabet stochastic processes.

Our procedure is based on {\em variable length Markov chain} models, also known as {\em context tree models}, that are generalizations of finite
order Markov chains. Context tree processes first appeared in Rissanen's seminal paper \cite{Rissanen1983}. Since then they have attracted considerable attention in different communities as they combine parsimonious models, interpretability and computational tractability \cite{BuhlmannWyner1999,Buhlmann2000,FerrariWyner2003,WillemsEtAl1995,Garivier2006,CsiszarTalata2006,TalataDuncan2009,CsiszarShields1996,Buhlmann1999,GarivierLeonardi2010}

We will use context trees to estimate transition probabilities, that is, next-symbol probabilities conditionally on the entire past. We also wish to select context trees that represent the dependency structure of the process. As in previous work, our main challenge is to select a statistically sound model from a huge and complex set of candidates.

Our approach builds upon the estimation method proposed in \cite{belloni2017approximate}. That method relies on building the full suffix tree for the observed sample, then pruning the tree by removing nodes that would improve the bias-variance tradeoff. A procedure reminiscent of Lepskii's adaptation method \cite{Lepskii1991} is used, which requires the specification of a node-specific regularization parameters. As in \cite{belloni2017approximate}, these parameters provide a measure of the ``variance"~in the estimation of transitions from a specific suffix. However, the choice of parameters in \cite{belloni2017approximate} was based on martingale concentration inequalities. While theoretically valid, these inequalities can be quite conservative: they lose constants and do not account correctly for correlations between sample events. 

Our main contribution to the VLMC estimation problem is to give a multiplier-bootstrap-based method for choosing the regularization parameters. This choice is less conservative than the approach via concentration, and thus leads to better estimation properties. The main ingredient needed to prove validity our our bootstrap is the the high-dimensional martingale CLT developed in Section \ref{SEC:MAX}.

We note that the connection between VLMCs, Markov chain order estimation, and martingale concentration has been known for some time. A classical paper by Csisz\'{a}r \cite{Csiszar2002} analyses MDL estimators for finite-order Markov chains via a finite-sample Law of the Iterated Logarithm for martingales. Since then, other estimators for Markov chains and VLMCs were analyzed via martingale techniques \cite{GarivierLeonardi2010,CsiszarTalata2006,belloni2017approximate,oliveira2015context}. We believe our martingale CLT and bootstrap may lead to tighter analyses of these estimators. In particular, our bootstrap may be useful for parameter tuning in these settings. 

\subsection{Organization}
The paper is organized as follows. In Section \ref{SEC:NOTATION} below we collect the notation we will use throughout the paper. Section \ref{SEC:CLTs} states our central limit theorems. Namely, a (quantitative) CLT for smooth functions of multidimensional martingales, performance bonds for a Gaussian bootstrap to approximate the expectation of smooth functions. Section \ref{SEC:MAX} derives a CLT for the maximum of a $d$-dimensional martingale, which can provide meaningful bounds even if $d\gg n$. This section also contains our new anti-concentration results. Finally, we give a construction of simultaneous confidence intervals for many means under (the martingale) dependence in Section \ref{SEC:MM} where the dimension of the martingale can exceed the sample size. Section \ref{SECCONTEXT} has our main application which is tuning penalty parameters for the estimation of VLMC models.

\subsection{Notation}\label{SEC:NOTATION}

In this work, $(M_t,\sF_t)_{t=0}^n$ is a martingale with $M_t\in\R^d$, $M_0=0$ and $\Ex{\|M_t\|^2}<+\infty$ for each $t\in [n]$ and some norm $\|\cdot\|$. We let $\xi_t:=M_{t}-M_{t-1}$ denote the increments of the martingale. We let $\Exp{t-1}{\cdot}:=\Ex{\cdot\mid\sF_{t-1}}$ denote conditional expectation with respect to $\sF_{t-1}$, so that $\Exp{t-1}{\xi_t}=0$ because of the martingale property. We define the random matrices
\[\Sigma_t:=\Exp{t-1}{\xi_t\xi_t'}\mbox{ and }V_n:=\sum_{t=1}^n\Sigma_t.\]

We assume that $N_1,\dots,N_n$ are i.i.d. standard $d$-dimensional random vectors defined in the same probability space as the martingale, but independent of $\sF_n$. We consider the sequence
\[\eta_t:=\Sigma_t^{1/2}\,N_t,\,1\leq t\leq n.\]
(Here the square root can be any random matrix that is $\sF_{t-1}$-measurable and also satisfies $\Sigma_ t^{1/2}\,(\Sigma^{1/2}_t)' = \Sigma_t$)).
Our goal will be to show that $M_n$ is close in distribution to $\sum_{t=1}^n\eta_t$. Note that the first two conditional moments of $\eta_t$ match those of $\xi_t$.
\[\Exp{t-1}{\eta_t} = 0\mbox{ and } \Exp{t-1}{\eta_t\eta_t'} = \Sigma_t.\]

We also assume that there exists a deterministic set $\sS\subset\R^d$ such that $\frac{\xi_t}{\|\xi_t\|}, \frac{\eta_t}{\|\eta_t\|}\in\sS.$ Our approximation bounds will be stated in terms of $\sS$.

The symbol ``$\preceq$" denote the positive semidefinite partial order on $d\times d$ symmetric matrices. For a norm $\|\cdot \|$ on matrices we have the associated dual norm:
\[\|B\|_*:=\sup_{A\in\R^{d\times d}\,:\,\|A\|=1}\,|\tr(BA)|,\]
so that $|\tr(AB)|\leq \|A\|\,\|B\|_*$ always.

Given $p\in [1,\infty]$, we use the symbol $\|\cdot\|_p$ to denote both the $\ell_p$ norm on vectors and the entrywise $\ell_p$ norm on matrices.

\section{CLT and bootstrap for multivariate martingales}\label{SEC:CLTs}

In this section we prove an useful approximation result for expectations $\Ex{\varphi(M_n)}$ of martingales $M_n$ that are allowed to have large dimension. We will keep our notation on matrices and norms and make an assumption on the quadratic variation $V_n$ of the martingale sequence. In what follows let $\delta_n$ be fixed.

\begin{assumption}\label{assump:A}
Let $(M_t,\sF_t)_{t=0}^n$ be a martingale as in Section \ref{SEC:NOTATION}. Assume there exists deterministic symmetric $d\times d$ matrices $V$, and $V_\delta\succeq 0$ such that:
\[\Pr{V_n\preceq V+V_\delta} \geq  1 - \alpha \ \ \mbox{and} \ \ \Delta_n := \|V_\delta\| + \Ex{\|V - V_n\|} \leq \delta_n.\]
\end{assumption}

Assumption \ref{assump:A} allows for many applications. The assumption that $\Delta_n \leq \delta_n \to 0$ trivially covers the case in which $V_n$ is deterministic. Importantly, it is a finite sample condition (i.e., $V$ can change with $n$) that covers modern applications in which the dimension $d$ of the martingale increases with the sample size $n$.

\begin{remark}Our assumption requires that $V_n$ is typically close to a deterministic $V$ and also that $V-V_n\preceq V_\delta$ for a  small~matrix $V_\delta$. It might be possible to avoid this condition using the predictable projection techniques of \cite{Grama1996,Rackauskas1995}. Unfortunately, we have not been able to do this effectively with the entrywise $\infty$ norm that we need to study the maximum.\end{remark}

We state next a first central limit theorem result for smooth functions of multivariate martingales. It will provide good results provided $\Delta_n$, as defined in Assumption \ref{assump:A}, is small. We consider a function $\varphi:\R^d\to\R$ which is three times differentiable. We let:
%
\[c_0:=\sup_{x,x'\in\R^d}|\varphi(x) - \varphi(x')|, \ \ \mbox{and} \ \ c_k:=\sup_{x\in \R^d,w\in\sS}\,|\nabla^k\varphi(x)\,(w^{\otimes k})|, \ \ k=2,3.\]

Theorem \ref{thm:lindeberg} below is based on a Lindeberg argument and characterizes the approximation errors based on $\alpha$, $\Delta_n$ and higher moments from the relevant increments.

\begin{theorem}[Quantitative CLT for Multidimensional Martingales]\label{thm:lindeberg} Suppose that Assumption \ref{assump:A} holds and let $\eta_1,\ldots,\eta_n$ be independent vectors in $\mathbb{R}^{d}$ with $\eta_t \sim N(0, \Sigma_t)$. Then for any function $\varphi:\R^d\to \R$ that is three times differentiable, we have:
\[|\Ex{\varphi(M_n)} - \Ex{\varphi(N(0,V))}|\leq c_0\alpha + 2c_2\Delta_n + c_3\sum_{t=1}^n\Ex{\|\xi_t\|^3 + \|\eta_t\|^3}.\]\end{theorem}

Theorem \ref{thm:lindeberg} is a nonasymptotic result. When the sample size $n$ grows, we will see that it allows for the dimension of the martingale to grow with $n$. How fast $d$ can grow while still allowing for a good approximation depends on the specific application, in particular on $c_3$ and the higher moments of the data. We will see that bounds on $c_3$ can be obtained by dual norms as well. Theorem \ref{thm:clt:max:martingale} below exploits an specific functional to allow potentially $d\gg n$.

The result above suggest we can compute functionals of $M_n$ based on simulation from a Gaussian random variable. However that requires the knowledge of the covariance matrix $V$, which might not be directly available. In many settings the quadratic variation $V_n$ or an estimator of it is computable instead. Therefore, based on a Gaussian perturbation lemma, we can derive the following corollary.

\begin{corollary}[Gaussian bootstrap for smooth functions]\label{cor:bootstrap} Suppose that Assumption \ref{assump:A} holds and adopt the notation in Section \ref{SEC:NOTATION}. Let $W_n$ be a random $d\times d$ positive semidefinite matrix that possibly depends on all the other randomness in our problem. Let $N(0,W_n)$ denote a $d$-dimensional random vector that is Gaussian with mean $0$ and covariance $W_n$, conditionally on all the other randomness in our problem. Then for any function three times differentiable function $\varphi:\R^d\to\R$ we have
\begin{eqnarray*}|\Ex{\varphi\left(M_n\right)} - \Ex{\varphi(N(0,W_n))\mid W_n}| &\leq& c_0 \alpha_n + 2c_2  \Delta_n + c_2\|V-W_n\| \\ & & + c_3 \sum_{t=1}^n\Ex{\|\xi_t\|^3 + \|\eta_t\|^3}.\end{eqnarray*}\end{corollary}

In some situations, the estimator $W_n$ can be taken to be $\sum_{t=1}^n \xi_t\xi_t'$. In that case it is possible to simulate a Gaussian random variable $N(0,V_n)$ via a Gaussian multiplier bootstrap procedure, conditional on $(\xi_t)_{t=1}^n$ we have
$$ M_n^* = \sum_{t=1}^n \xi_t g_t$$
where the $(g_t)$ are i.i.d. standard Gaussian random variables generated independently from the data.

\section{A Central Limit Theorem for the Maximum of a Multivariate Martingale}\label{SEC:MAX}

We highlight a set of applications in which we are interested on confidence regions. In those applications, the relevant functional is no longer smooth and can include indicator functions. The next result builds upon Theorem \ref{thm:lindeberg} and recent Gaussian coupling for the maximum of the sum of independent random vectors established in \cite{chernozhukov2012gaussian,chernozhukov2012comparison,chernozhukov2013gaussian,chernozhukov2015noncenteredprocesses}.
In that case a smooth approximation of the maximum $F_\beta(X)=\beta^{-1}\log( \sum_{j=1}^d \exp(\beta X_j))$ and a smooth approximation of indicator functions will have well behaved derivatives. The result relies on specific choices of norms, namely $\|\cdot\|=\|\cdot\|_\infty$ and $\|\cdot\|_* = \|\cdot\|_1$.

\begin{theorem}\label{thm:clt:max:martingale}
Suppose that Assumption \ref{assump:A} holds. Let $H$ be a $d$-dimensional Gaussian vector with mean $0$ and covariance matrix $V$. Also let $W_n$ denote a random $d\times d$ positive semidefinite random matrix (potentially depending on other randomness in the problem), and consider a random vector $H_n$ that $d$-dimensional Gaussian with mean $0$ and covariance $W_n$, conditionally on all the other random variables. Let $$Z:=\max_{j\in[d]}\sum_{t=1}^n\xi_{tj},\;\; \widetilde Z :=  \max_{j\in[d]}H_j,\;\;\mbox{and}\;\;\widetilde Z_n:=\max_{j\in[d]}H_{n,j}.$$ Define:
\[a_n(\delta):=2\alpha +  \frac{C\log d}{\bar\delta^2} \Delta_n +\frac{C\log^2d}{\bar \delta^3} \sum_{t=1}^n \Ex{\|\xi_t\|_\infty^3 + \|\eta_t\|_\infty^3},\]
where $C$ is a universal positive constant. Then for every $\bar\delta>0$ and every Borel subset of $\mathbb{R}$ we have:
$$\Pr{ Z \in A} \leq \Pr{ \widetilde Z \in A^{C\bar \delta}} + a_n(\bar \delta)$$
and
$$\Pr{ Z \in A} \leq \Pr{ \left. \widetilde Z_n \in A^{C\bar \delta}\right| W_n}+ \frac{C\log d}{\bar\delta^2} \|V-V_n\|_\infty + a_n(\bar \delta).$$
\end{theorem}

Combined with Stressen's theorem, Theorem \ref{thm:clt:max:martingale} provides a coupling result for the maximum of many martingale sequences and the maximum of a Gaussian process up to small approximation errors.

\subsection{Anti-concentration for the Maximum of Gaussian Vectors}

As discussed in the literature for the independent case, the maximum of a high-dimensional vector tends to concentrate. However, anti-concentration results have been established to show that the maximum of a Gaussian vector cannot concentrate too fast around any point, see e.g. \cite{chernozhukov2012gaussian} and \cite{chernozhukov2015noncenteredprocesses}. In what follows we derive another anti-concentration result for centered Gaussian vectors with components that have different variances. In contrast to the literature, the bound depends only on the maximum variance (in opposition to the minimum variance) but we incur an additional log factor. This seems of independent interest and relevant in our application to context trees in Section \ref{SECCONTEXT} since it is possible to have components with arbitrary small variances (even zero). A notable exception to the literature is a very recent result in \cite{lopes2018bootstrapping} for the anti-concentration of the maximum of a Gaussian vector in which the variances of components, while potentially decaying to zero, have a suitable structure. Below we state our anti-concentration result.

\begin{theorem}\label{thm:max:anticoncentration}
Let $X \in \mathbb{R}^d$ be a zero-mean Gaussian random vector, $d\geq 2$. Define the maximum variance $\bar \sigma^2 = \max_{j\in[d]} \Exp{}{X_j^2}>0$. Then, for any $t \geq \bar \sigma\Phi^{-1}(0.95)$ and $\epsilon < \frac{1}{8}t$, we have
$$   \Pr{ \left|\max_{j\in[d]} X_j - t \right|\leq \epsilon } \leq 2\epsilon \frac{\sqrt{\log d}}{\bar \sigma}\{\sqrt{2\log (2d)}+3\}.$$
\end{theorem}

The anti-concentration result in Theorem \ref{thm:max:anticoncentration} explicitly makes use of a ``large value" of $t$ instead of considering $t \in \mathbb{R}$. This is useful in applications where we are concerned with a high quantile. This allows us to remove the impact of components with small variance, as they are unlikely to realize the maximum near a large value of $t$.

The following corollary combines the result of Theorem \ref{thm:max:anticoncentration} with Lemma 4.3 in \cite{chernozhukov2015noncenteredprocesses}.

\begin{corollary}\label{cor:anti-concentration}
Let $X \in \mathbb{R}^d$ be a zero-mean Gaussian random vector, $d\geq 2$. Let $\bar \sigma^2 = \max_{j\in[d]}{\rm Var}(X_j)$, $\underline{\sigma}^2 = \min_{j\in[d]}{\rm Var}(X_j)$. Then for every $\varepsilon \in (0, \frac{1}{4}\bar\sigma)$
$$ \sup_{t \geq 2\bar \sigma} \Pr{ \left| \max_{j\in[d]} X_j - t\right| \leq \epsilon } \leq  2\epsilon \{\sqrt{2\log (2d)}+3\} \min\left\{ \frac{1}{\underline{\sigma}}, \ \  \frac{\sqrt{\log d}}{\bar \sigma}  \right\}$$
\end{corollary}

Corollary \ref{cor:anti-concentration} provides a more complete picture. Moreover, it provides us with a transition between the results of the literature and ours as $\underline{\sigma}$ goes to zero. These results were shown for centered Gaussian random vectors. However, Lemma 4.3 in \cite{chernozhukov2015noncenteredprocesses} also holds for non-centered Gaussian random vectors. We note that further restrictions on $t$ based on the centering can lead to a anti-concentration bounds that could be useful.


\subsection{Simultaneous Confidence Intervals for Many Means}\label{SEC:MM}

A basic application of the results developed here is the construction of simultaneous confidence intervals for means of high-dimensional martingales. Recently this problem has attracted a lot of attention under independence, see \cite{belloni2018high} for a survey, and under other time dependence structures \cite{chernozhukov2013testing,zhang2014bootstrapping,zhang2017gaussian}. For example, as discussed in \cite{pauly2011weighted}, in analyzing comparative gains in financial applications. Let $r_{kj}$ denote the comparative log-gains (or comparative log-returns) of the $j$th asset from time $k-1$ to $k$. To model such applications we consider the model
$$ r_{kj} = \mu_j + X_{kj}$$
where we observe $r_{kj}\in\R$, for $k\in[n],j\in[d]$. for each $j\in [d]$ we have that $X_{kj}, k\in [n],$ is a martingale difference and $\mu_j \in \mathbb{R}$ is an unknown value of interest (i.e. mean comparative log-gains of the $j$th asset).

Since we observe $r_{kj}, j\in[d], k\in [n]$, we can directly compute \begin{equation}\label{def:MM}\hat \mu_j = \frac{1}{n}\sum_{k=1}^n r_{kj} \ \ \ \mbox{and} \ \  \hat Z_{kj} = \hat \sigma_j^{-1}(r_{kj} - \hat \mu_j) \end{equation} where $\hat\sigma_j^2 = \frac{1}{n}\sum_{k=1}^n (r_{kj} - \hat \mu_j)^2$. For convenience, in what follows let $Z_{kj} =  \sigma_j^{-1}X_{kj}$.

We are interested in constructing confidence regions for $\mu \in \mathbb{d}$ where $d$ is large. As it has been recently discussed in the literature, when considering many means, i.e. $d$ grows with $n$, the shape of the confidence regions considered plays a key role in the analysis and validity of the construction. In what follows we provide two set of results for simultaneous inference.

The first pertains to simultaneous confidence bands of the form
\begin{equation}\label{def:SCB}
\hat\mu_j - \widehat{\rm cv}(\delta)\frac{\hat\sigma_j}{\sqrt{n}} \leq \mu_j \leq \hat\mu_j + \widehat{\rm cv}(\delta)\frac{\hat\sigma_j}{\sqrt{n}} \ \ \ j\in [d]
\end{equation}
where $ \widehat{\rm cv}(\delta)$ is chosen for the relation (\ref{def:SCB}) to hold with probability $1-\delta - o(1)$. We will set
$$  \widehat{\rm cv}(\delta) = \mbox{conditional $(1-\delta)$-quantile of} \ \ \widehat{Z}^* = \max_{j\in[d]} \left| \frac{1}{\sqrt{n}}\sum_{k=1}^n g_k \hat Z_{kj}\right| \ \ \mbox{given} \ (\hat Z_k)_{k=1}^n $$
where $(g_k)_{k\in[d]}$ are i.i.d. standard Gaussian random variables independent of the data.

\begin{assumption}\label{assum:MMinfty}
Let $\delta_n$ and $\psi_n$ be fixed sequences going to zero with $\psi_n = o(\delta)$. Then:\\
(i) for some $\rho \in (0,1/2)$ we have  $\frac{1}{n}\sum_{k=1}^n\Exp{}{\|Z_k\|_\infty^3}\leq n^\rho$;\\
(ii)  $V = \frac{1}{n}\sum_{k=1}^n \Exp{}{Z_{k}Z_{k}'}$ and (the random matrix) $V_n = \frac{1}{n}\sum_{k=1}^n \Exp{k-1}{Z_{k}Z_{k}'}$ satisfy $$ \max_{j,\ell \in [d]}|V_{j\ell}| \leq C \ \ \mbox{and} \ \ \Exp{}{\max_{j,\ell\in[d]}|V_{n,j\ell}-V_{j\ell}|}\leq \delta_n/\log^2(dn),$$ and  with probability $1-\psi_n$, we have $$ V_n \preceq \{1+ \delta_n/\log^2(d)\}  V, \  \  \ \mbox{and}\ \ \ \max_{j,\ell\in [d]} \left| V_{n,k\ell} - V_{k\ell}\right|  \leq \delta_n/\log^2(dn)$$
(iii)   with probability $1-\psi_n$ $$\max_{j,\ell\in [d]} \left| \frac{1}{n}\sum_{k=1}^n Z_{kj}Z_{k\ell} - \Exp{k-1}{Z_{kj}Z_{k\ell}} \right|  \vee \max_{j\in [d]} \left| \frac{1}{n}\sum_{k=1}^n (\hat Z_{kj}-Z_{k j})^2 \right| \leq \delta_n^2/\log^2(dn)$$
(iv) $\psi_n\leq \delta^2\delta_n^{1/2}$, $\delta_n^{1/20} \leq \delta$,  $\log^{3.5} d \leq \delta_n^2 n^{1/2-\rho}$, and $\min_{j\in[d]} \sigma_j \geq c$.
\end{assumption}

Assumption \ref{assum:MMinfty}(i) imposes moment conditions. Assumption \ref{assum:MMinfty} (ii) imposes some restrictions on the correlation structure in order. Indeed it allows for block diagonal matrices as long as the block size does not grow too fast.  Assumption \ref{assum:MMinfty}(iii) is a weak assumption. The first quantity can be bounded via many known martingale inequalities while the second, although it is context dependent, follows by using standard plug-in rule to compute $\hat Z_{kj}$. Finally, Assumption \ref{assum:MMinfty}(iv) imposes a trade off between how fast the number of components $d$ can grow relative to the sample size $n$. It also imposes that the

\begin{theorem}\label{thm:MMinfty}
Suppose that Assumption \ref{assum:MMinfty} holds and $d\geq 2$. Then we have
$$ \left|\Pr{\hat\mu_j - \widehat{\rm cv}(\delta)\frac{\hat\sigma_j}{\sqrt{n}} \leq \mu_j \leq \hat\mu_j + \widehat{\rm cv}(\delta)\frac{\hat\sigma_j}{\sqrt{n}} \ \ \ \mbox{for all} \ j\in [d] } - (1-\delta)\right| \leq o(\delta) $$
\end{theorem}

Theorem \ref{thm:MMinfty} establishes the validity of the simultaneous confidence intervals. We note that in addition to the dependence it allows for $d\gg n$. This is an interesting feature to handle modern high-dimensional applications. The proof builds upon Theorem \ref{thm:clt:max:martingale} and anti-concentration arguments used recently in the literature. By exploiting the structure of the maximum, we can reduce drastically the requirements on $d$ relative to $n$ from polynomial to logarithmic.

%
%
%
%

\section{Tuning many parameters in VLMC estimators}\label{SECCONTEXT}

In this section we leverage the tools developed in Section \ref{SEC:CLTs} to the estimation of a context tree associated with a discrete stationary stochastic process. Specifically, we use Theorem \ref{thm:clt:max:martingale} to develop data-driven tuning parameters that will be theoretically valid. In this example, the dimension of the martingale grows with the sample size to better balance bias and variance of the estimator.

\subsection{Preliminaries}

We follow \cite{belloni2017approximate}. $(X_k)_{k\in\Z}$ is a stationary stochastic process taking values in the finite alphabet $A$. For $k\in\N$, $A^{-1}_{-k}$ is the set of all strings  of length $k$ over $A$, which we index by the numbers $-k$, $-k+1$, $\dots$, $-1$ from left to right.

$A^*=\cup_{k=0}^{-\infty}A^{-1}_{-k}$ is the set of finite strings over $A$, where $A^{-1}_{-0}$ consists solely of the empty string $e$. For $w\in A^*$, $|w|$ is the length of $A$. This notation extends to the set $A^{-1}_{-\infty}$ of infinite strings indexed by the negative integers.

Given a finite of infinite string $w$ over $A$ that is indexed by the negative integers, and a number $0\leq k\leq |w|$, we let $w^{-1}_{-k}$ denote the suffix of length $k$ of $w$. We define a partial order on strings by saying that $w\preceq w'$ if $w$ is a suffix of $w'$. The empty string is the unique minimum element of this partial order.

Given $w\in A^*$, we let $\pi(w):=\Pr{X^{-1}_{-|w|}=w}$. The support $\supp\subset A^*$ is the set of $w\in A^*$ with $\pi(w)>0$. For $w\in \supp$ and $a\in A$, we define the transition probability:
\[p(a|w) = \frac{\pi(wa)}{\pi(w)} = \Pr{X_0=a\mid X^{-1}_{-|w|}=w}.\]
We extend this notation to $x=x^{-1}_{-\infty}\in A^{-1}_{-\infty}$ via the usual measure-theoretic defintion of probabilities. We define the continuity rate at $w\in \supp$ as:
\begin{equation}\gamma(w):=\sup\limits_{w',w''\in \supp, w',w''\succeq w;a\in A}|p(a|w') - p(a|w'')|.\end{equation}

We will be especially interested in processes with continuous transition probabilities, for which $\gamma(x^{-1}_{-k})\to 0$ as $k\to +\infty$ for every $x^{-1}_{-\infty}\in A^{-1}_{-\infty}$. Our goal will be to obtain estimates of the transition probabilities that adapt to the continuity rates.

The {\em parent} ${\rm par}(w)$ of $w\in A^*\backslash\{e\}$ is the suffix of $w$ of length $|w|-1$. A child of $w$ is a $u\in A^*$ with ${\rm par}(u)=w$.

A nonempty subset $T\subset A^*$ is a tree if for all $w\in T\backslash\{e\}$, ${\rm par}(w)\in T$. A node $w\in T$ is a leaf of $T$ if none of its children belong to $T$. $T$ is complete if all nodes are either leaves or have exactly $|A|$ children. $\Leaves{T}$ is the set of leaves of $T$.

For a finite tree $T$ and $x\in A^{-1}_{-\infty}$, we let $T(x)$ denote the largest suffix of $x$ in $T$. $T$ is said to be a {\em context tree} for the process $(X_k)_{k\in\Z}$ is said to have context tree $T$ if $p(a|X^{-1}_{-\infty}) = p(a|T(X^{-1}_{-\infty}))$ almost surely.


\subsection{A context-tree-based estimator}

We describe next the estimation procedure proposed in \cite{belloni2017approximate} with a few (trivial) modifications.

We assume we are given a sample $X_1^n=(X_1,\dots,X_n)$ of size $n$ of the process $(X_k)_{k\in \Z}$. Fix a parameter $h_*\in\N$. For a finite string $w\in A^*$ with $|w|\leq h_*$, and $a\in A$, define:
\[N_{n}(w):= \sum_{k=h_*}^{n}\Ind{\{X^{k}_{k-|w|+1}=w\}},\]
and
\[N_{n-1}(w):= \sum_{a\in A}N_{n}(wa).\]
Note that the sum starts with $k=h_*$, so that $N_{n-1}(x) = \sum_{b\in A}N_{n-1}(bx)$ for all $a\in A$ (otherwise it could be that $x$ occurs at positions $1,\dots,|x|$).

The {\em empirical tree} of the sample is the set $E_n\subset A^*$ of all $w\in A^*$ with $|w|\leq h_*$ and $N_{n-1}(w)>0$. For such $w$ we set:
\begin{equation}\label{def:np}\widehat{p}_n(a|w):= \frac{N_{n}(wa)}{N_{n-1}(w)} \ \ \mbox{and} \ \ \overline{p}_n(a|w):= \frac{\sum_{k=h_*}^{n}p(a|X^{k-1}_{1})\,\Ind{\{X^{k-1}_{k-|w|}=w\}}}{N_{n-1}(w)}.\end{equation}
Note that $\hat p_n$ is the usual nonparamentric estimate for the transition probabilities. Quantity $\bar p_n$ is sample-dependent, but is guaranteed to satisfy:
\[\forall w'\succeq w\,:\,|\bar p_n(a|w) - p(a|w')|\leq \gamma(w).\]

The estimator from \cite{belloni2017approximate} is defined as follows. Fix a constant $c>1$. For each $w\in E_n$ one has defined a parameter ${\rm cf}(w)$. Intuitively, ${\rm cf}(w)$ is measure of the deviations $\max_{a\in A}|\hat p_n(a|w') - \bar p_n(a|w)|$. For the time being, all we need is that computed from the sample $w$ and is increasing in the partial order $\preceq$.

Now define, for each such $w$, a number ${\sf CanRmv}(w)\in \{0,1\}$ such that ${\sf CanRmv}(e)=0$ for the empty string and, for $w\in E_n\backslash\{e\}$,
\[{\sf CanRmv}(w)=1\Leftrightarrow \left\{\begin{array}{l}\forall w',w''\in E_n\mbox{ with }w'\succeq w, w''\succeq {\rm par}(w):\\ \max_{a\in A}|\hat p_n(a|w') - \hat p_n(a|w'')|\leq c\,({\rm cf}(w') + {\rm cf}(w'')).\end{array}\right.\]

We define a subset $\hat T_n\subset E_n$ as follows:
\[\hat T_n:= \{w\in E_n\,:\, {\sf CanRmv}(w)=0\}.\]
This set is a tree because if ${\sf CanRmv}(w)=0$, then ${\sf CanRmv}({\rm par}(w))=0$ as well.
consisting of all $w\in T_n$ that satisfy ${\sf CanRmv}(w)=0$ is a tree. We take this tree to be our estimator of the context tree of the process. Our estimate for the transition probabilities is defined in terms of $\hat T_n$.
\[\hat P_n(a|x):= \hat p_n(a|\hat T_n(x))\,\,(x\in A^{-1}_{-\infty},a\in A).\]

The next Theorem is essentially contained in \cite[Lemmas A.1 and A.2]{belloni2017approximate}. Notice that it is a deterministic statement.

\begin{theorem}Assume the following event holds.
\begin{equation}\label{eq:defGood}{\rm Good}_*:=\{\forall (w,a)\in E_n\times A\,:\, |\bar p_n(a\mid w) - \hat p_n(a|w)|\leq {\rm cf}(w)\}\end{equation}
Then, $\hat T_n$ is contained in the true context tree $T^*$ of the process. Moreover, for almost all realizations $x\in A^{-1}_{-\infty}$ of $X_{-\infty}^{-1}$, and all $a\in A$:
\[|\hat P_n(a|x) - p(a|x)|\leq \inf_{w\in E_n,w\preceq x}\left( \frac{2c+2}{c-1}\,\gamma(w) + (1+2c)\,{\rm cf}(w)\right).\]
\end{theorem}

We note that this result implies an oracle inequality for estimating transition probabilities of $\beta$-mixing processes (\cite[Theorem 2]{belloni2017approximate}).

\subsection{Bootstrap-based choice of confidence radii}

In order for our estimator to work well, we need choices of ${\rm cf}(w)$ that ensure that the event ${\rm Good}_*$ holds with high probability. In other words, we need that ${\rm cf}(w)$ are good ``confidence radii", in the sense that the condition $|\bar p_n(a\mid w) - \hat p_n(a|w)|\leq {\rm cf}(w)$ holds for all $w\in E_n$ and $a\in A$ at a prescribed confidence level.

The paper \cite{belloni2017approximate} proposes a conservative choice for the parameters ${\rm cf}(w)$, which is based on a martingale concentration inequalities.
\begin{equation}\label{eq:old}  {\rm cf}(w) \equiv \sqrt{\frac{4}{N_{n-1}(w)}\left(2\log(2+\log_2 N_{n-1}(w))+\log(n^2|A|/\delta)\right) }\end{equation}
This choice does not take into account the correlation structure of the differences $\bar p_n(a\mid w) - \hat p_n(a|w)$ for different $a$ and $w$.

We propose here a different, data-driven choice of ${\rm cf}(w)$ that is based on our martingale bootstrap. We define a martingale $M_n\in\R^{\supp}\otimes \R^A$ with coordinates \[\begin{array}{rl}
M_n(w,a) & \displaystyle :=\frac{N_{n-1}(w)\,(\widehat{p}_n(a|w) - \bar{p}_n(a|w))}{\sqrt{\pi(w)\,n}}\\
&\displaystyle  = \frac{\sum_{k=h_*}^{n}[\Ind{\{X_k=a\}} - p(a|X^{k-1}_{1})]\,\Ind{\{X^{k-1}_{k-|w|}=w\}}}{\sqrt{\pi(w)n}}\,\,((w,a)\in \supp\times A).\end{array}\]

In principle, an ideal critical value would be defined by normalizing the components of $M_n$ to have unit variance as set $$ {\rm cv}^*(\delta) = \mbox{ $(1-\delta)$-quantile of} \ \ \max_{w\in E_n, a \in A} | M_n(w,a) |/\sqrt{\bar p_n(a\mid w)(1-\bar p_n(a\mid w))}$$
which leads to the following definition of ${\rm cf}$ for each $w\in E_n$
$$ {\rm cf}(w) \equiv {\rm cv}^*(\delta)  \frac{\sqrt{\max_{a\in A}\bar p(a\mid w)\{1-\bar p(a\mid w)\}\pi(w)n}}{N_{n-1}(w)}. $$ By definition of the quantile we have that (\ref{eq:Property}) holds. Unfortunately, the tree $E_n$ is random and for deep nodes it would be hard to estimate $\pi(w)$ and $\bar p(a\mid w)$ reliably. Indeed those can be estimated reliably only on typical subtrees as discussed in the typicality assumption. Moreover, the term $\bar p_n(a\mid w)\{1-\bar p_n(a\mid w)\}$ can be arbitrary small (or even zero) for a given $w\in E_n$ and $a\in A$ in some applications.

To mitigate these issues we use a different construction for the critical value. We define the martingale difference for $w \in E_n, a \in A, k \in [n]$:
$$d_k(w,a) = \frac{1}{\sqrt{\pi(w)n}}\{\Ind{\{X_k=a\}} - p(a\mid X^{k-1}_{-\infty})\}\Ind{\{X^{k-1}_{k-|w|}=w\}}$$ and the associated estimator $$\hat d_k(w,a) = \frac{1}{\sqrt{N_{n-1}(w)}}\{\Ind{\{X_k=a\}} - \hat p(a\mid w)\}\Ind{\{X^{k-1}_{k-|w|}=w\}}.$$
Given a finite tree $T\subset A^*$, the estimate for the critical value is computed as
$$\widehat{{\rm cv}}(\delta) = \mbox{conditional} \  (1-\delta)\mbox{-quantile of} \ \max_{w\in T, a\in A}\sum_{k=1}^n g_k \hat d_k(w,a)$$
where $(g_k)$ are i.i.d. standard Gaussian random variables. We then set:
\[\widehat{{\rm cf}}(w):=\frac{\widehat{{\rm cv}}(\delta) }{\sqrt{N_{n-1}(w)}}\mbox{ if }w\in T,\]
and for $w \not \in T$ we set $ {\rm cf}(w)$ as in (\ref{eq:old}) with $\delta/n$ instead of $\delta$.

\begin{remark}[On the definition of the tree $T$] The choice of tree $T$ is part of its definition of the parameters ${\rm cf}(w)$ and is thus part of the definition of the estimator. For simplicity, we assume that $T$ is chosen before seeing the data. Some modifications allow for data-driven choices of $T$; for instance, $T$ may consist of all $w\in E_n$ with $N_{n-1}(w)\geq t_n$ for some $t_n$. On the other hand, simulations suggest that taking $T=E_n$ is not a valid choice. \end{remark}

\begin{remark}[Implementation and Computational Aspects of the Bootstrap]
In many settings, bootstrap procedures are computationally intensive and not efficient for high-dimensional applications. No such problem arises in our case, as the procedure benefits from a recursive property related to the tree structure. Indeed, this property is typical of many other estimators for context trees in the literature. 

Specifically, the calculation of the bootstrapped quantities can be performed recursively. Given a bootstrap replication with multipliers $(g_k)_{k=1}^n$, letting \[N_n^*(w) = \sum_{k=1}^n g_k \Ind{\{X^k_{k-|w|}=w\}},\] we have $N_n^*(w)=\sum_{b\in A}N_n^*(bw)$ and
 $$ \sum_{k=1}^ng_k\hat d_k(w,a) = \frac{N^*_{n-1}(wa) - \hat p(a\mid w) N^*_{n-1}(w)}{\sqrt{N_{n-1}(w)}}. $$
Therefore the recursion is similar to the recursion already present in many context tree algorithms. Thus, after the appropriate data structure is set, after the construction of the tree $E_n$, computing the bootstrap repetitions consists of aggregating the multipliers through the tree from the leaves to the root recursively (no recalculation of the tree $E_n$ is needed). $\square$ \end{remark}

\subsection{Conditions for validity of the bootstrap}

We now analyse the validity of our bootstrap-based method. We work in a setting where the sample size $n$ grows and the tree $T=T_n$ potentially depends on $n$ (although this dependence will be left implicit). We will need to quantify the {\em typicality} and {\em continuity} parameters of $T$.
\begin{itemize}
\item[] {\em Typicality:} Given $\eps\in (0,1)$, we assume $\alpha_\eps,\eps\in (0,1)$ is such that
\[\Pr{\forall w\in T\,:\,\left|\frac{N_{n-1}(w)}{\pi(w)\,n} - 1\right|\leq \eps}\geq 1-\alpha_\eps.\]
\item[]{\em Continuity at the leaves: } we let $\gamma>0$ be such that, for all leaves $w\in \Leaves{T}$, the continuity parameter at $w$ satisfies $\gamma(w)\leq \gamma$.\end{itemize}

These numbers characterize the how the tree $T$ can be used as (an approximate) context for the process  $(X_k)_{k\in\Z}$. The typicality property bounds how big $T$ can be for a desired precision $\epsilon$ and confidence $1-\alpha$. On the other hand, to have a small bias for approximating context longer than its leaves, $T$ might not be too small. Thus every choice of context $T$ can be associated with a triple $(\epsilon,\alpha_\eps,\gamma)$.  Clearly, if $T$ is the exact context tree, we have $\gamma = 0$, however in many settings we want to consider trees that are adaptive to the sample size (and potentially grow). We refer to \cite{belloni2017approximate} for simple conditions on the process  $(X_k)_{k\in\Z}$ and choices of trees $T=T_n$ that lead to specific triples $(\epsilon,\alpha_\epsilon,\gamma)\to (0,0,0)$ as $n\to \infty$.

Next we state the exact assumptions we need for our method to work. 

\begin{assumption}\label{assump:C}
Suppose that the following conditions hold:\\
(i) for some $\rho \in (0,1/2)$ we have  $\sum_{w \in \Leaves{T}} \pi^{-1/2}(w) \leq Cn^{\rho} $\\
(ii) $\max_{a\in A, w\in \Leaves{T}} \bar p^{1/2}(a\mid w)\{1 - \bar p(a\mid w)\} \geq c $\\
(iii) for some sequence $\delta_n \to 0$, the following relations hold:\\
$ \epsilon \leq  \delta_n\frac{\delta}{\log (d) \log^{1/2}(dn/\delta)\log (n)}
$, \ \
$\alpha_\eps + \gamma|A| \leq \delta_n\frac{\delta^3}{\log^2(n)\log^2(d)}$, \ \ $ n^{-1/2+\rho} \leq \delta_n\frac{\delta^4}{\log^5(d)\log^4(n)}$\\
\end{assumption}

\begin{figure}[h]
\begin{center}
\begin{tabular}{p{0.49\textwidth}p{0.47\textwidth}}
\centering Penalty based on bootstrap & \centering Penalty based on self-normalizion \end{tabular}
\includegraphics[width=0.49\textwidth]{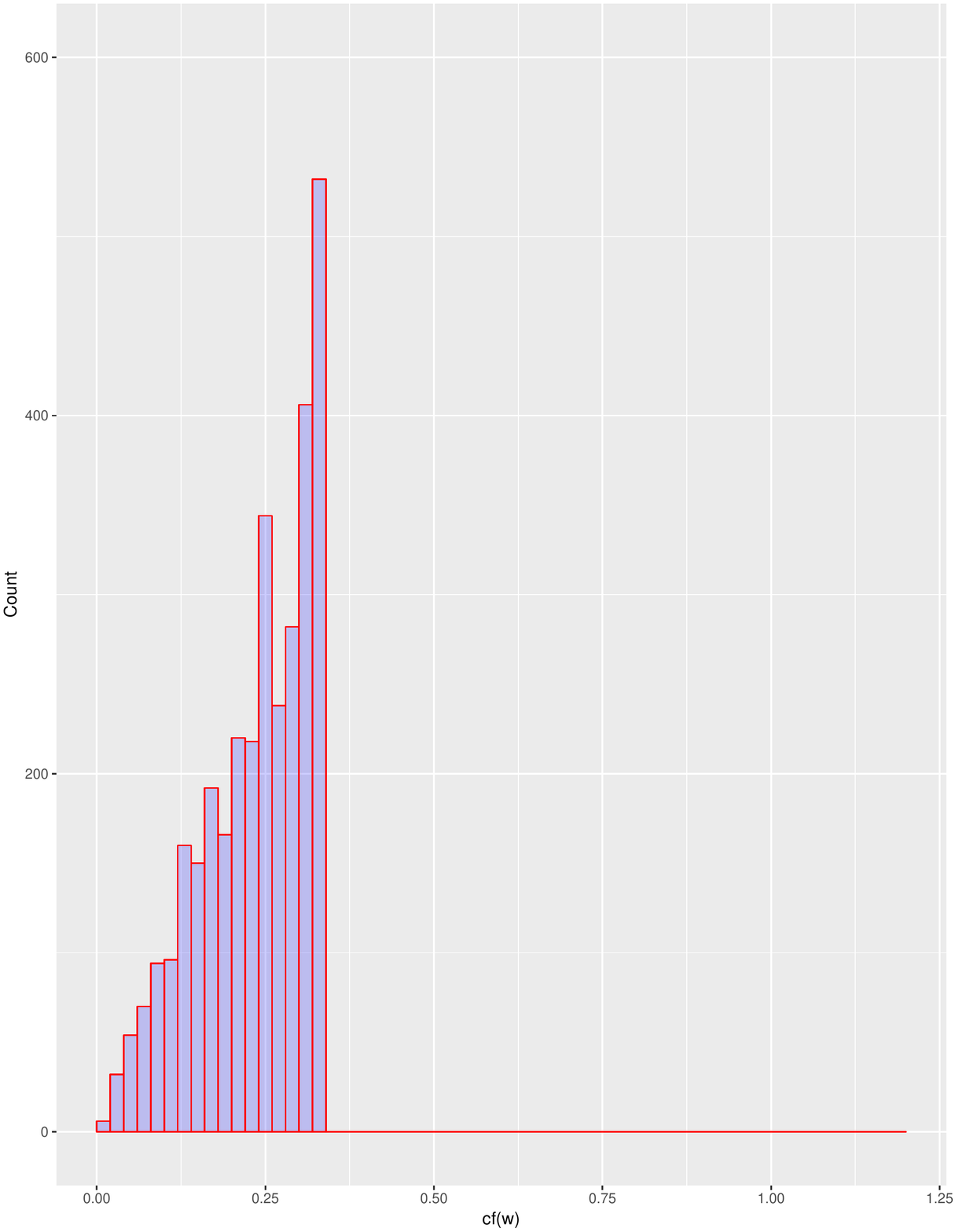}%
\includegraphics[width=0.49\textwidth]{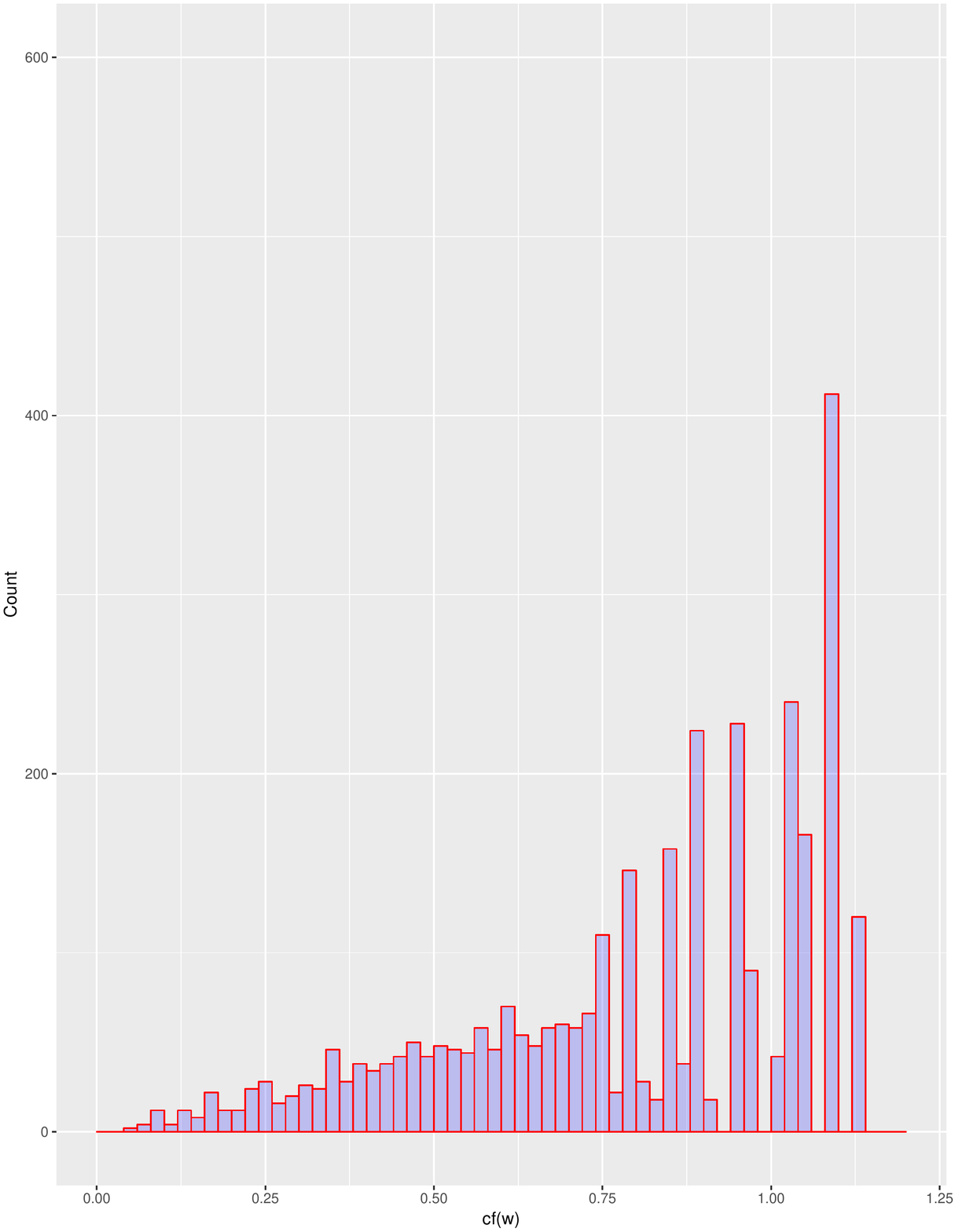}%
\caption{We considered a Markov chain of order 3 and a sample size of $n=5000$. The figure illustrates the histogram of the 1630 penalty parameters ${\rm cf}(w)$, $w\in E_n$, for the proposed bootstrap based methods (left) and the analytical bounds based on self-normalization (right). Both choices are theoretically valid but the bootstrap based adapts to the correlation structure of the process leading to smaller penalty choices.}
\end{center}
\end{figure}

\begin{figure}[h]
\begin{center}
\includegraphics[width=0.8\textwidth, height=0.5\textheight]{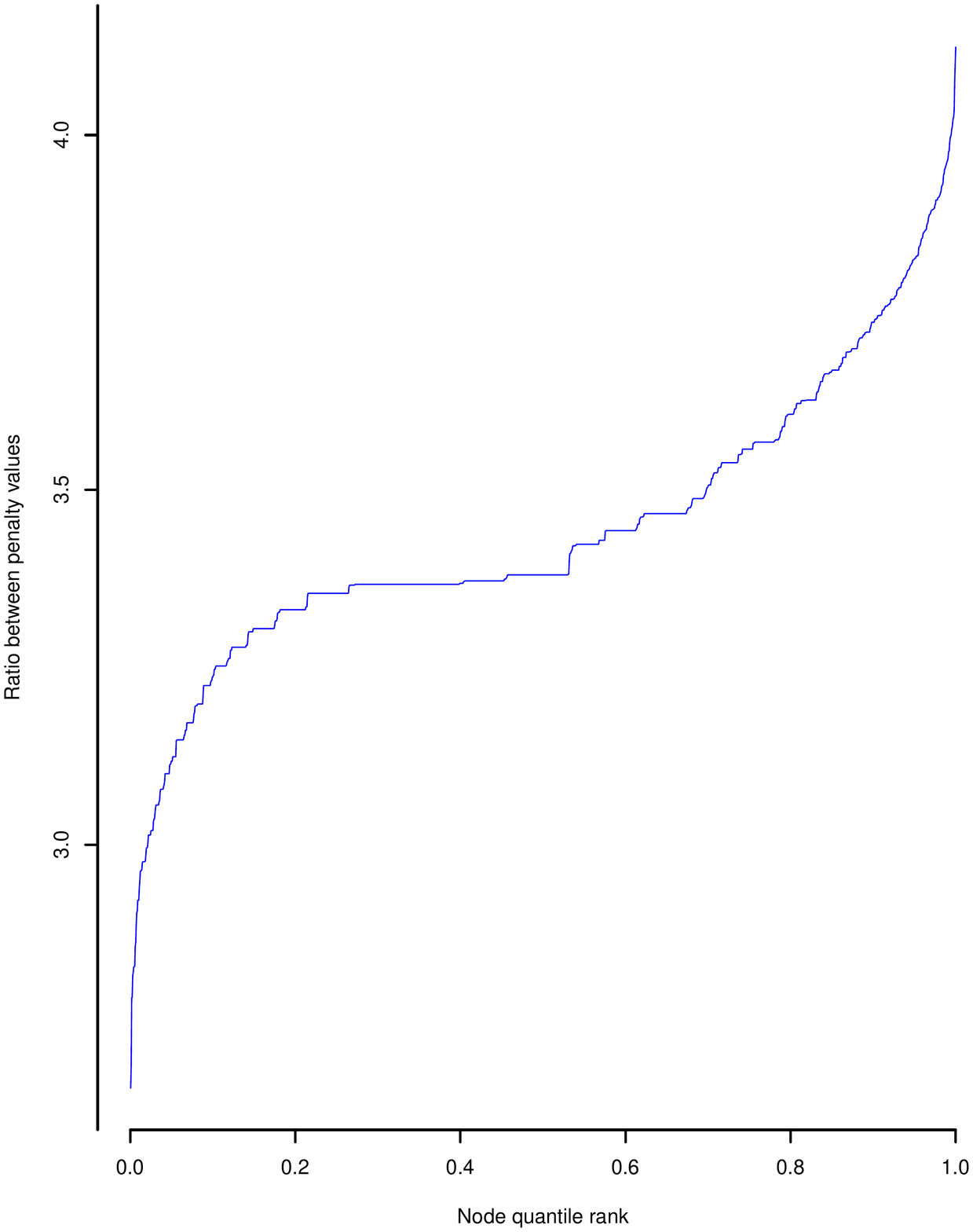}%
\caption{We considered a Markov chain of order 3 and a sample size of $n=5000$. The figure illustrates the penalty parameters ${\rm cf}(w)$, $w\in E_n$ for the different proposals. For each node we computed the ratio (self-normalized choice divided by the bootstrap-based choice) and ordered these ration to be plotted. The bootstrap based proposal seems to be a factor of three smaller across nodes.}
\end{center}
\end{figure}

Assumption \ref{assump:C})(i) allows for the frequency of the leaves of $T$ to decrease to zero but it bounds how fast it can decrease to zero. Note that it bounds the number of leaves of $T$ to be bounded by $Cn^\rho$. Assumption \ref{assump:C})(ii) is very mild as it requires only that at least one transition probability to be bounded away from zero and from 1. Finally, Assumption \ref{assump:C})(iii) provides sufficient conditions relating $n$, the dimension $d$ and other parameters. In particular it allows for $d$ to increase with $n$. Our result is as follows. 


\begin{theorem}\label{thm:ContextFinal}
Suppose that Assumption \ref{assump:C}, typicality and continuity hold. Then the event ${\rm Good}_*$ defined in (\ref{eq:defGood}) has probability:
\[\Pr{{\rm Good}_*}\geq 1-\delta - o(\delta/\log n).\]
In particular, we obtain that, with probability $\geq 1-\delta - o(\delta/\log n)$, for all $a\in A$ and $x\in A^{-1}_{-\infty}$,
\[|\hat P_n(a|x) - p(a|x)|\leq \inf_{w\in T,w\preceq x}\left( \frac{2c+2}{c-1}\,\gamma(w) + (1+2c)\,{\rm cf}(w)\right).\]\end{theorem}

Theorem \ref{thm:ContextFinal} characterizes sufficient conditions for the use of the proposed $\widehat{{\rm cf}}(w)$ that uses a bootstrap procedure. This result builds upon the general  theorem for the maximum of a high-dimensional martingale, and the new anti-concentration result. However, the control of various approximation errors relies on the structure of the context trees. In the next section, we collect the main ideas and technical results that are also used in the proof of Theorem \ref{thm:ContextFinal}.

\subsection{Leaf martingales, the operator $\Sum$, and quadratic variation}

It will be convenient to consider the simpler object $\Leaves{M}_n\in\R^{\Leaves{T}}\otimes \R^A$ obtained by restricting $M_n$ to $(w,a)$ with $w \in \Leaves{T}$. We call $\Leaves{M}_n$ and $M_n$ the {\em  leaf martingale} and {\em full martingale}, respectively.

The two martingales are related by a linear operator that we will now describe. Given $w\in \Leaves{T}$, let ${\rm Path}_T(w)$ denote the set of all $\tilde w\in T$ with $\tilde w\preceq w$ (i.e., $\tilde w$ that lies on the path between $w$ and the root of $T$). Let $\{e_w\}_{w\in \Leaves{T}}$ denote the canonical basis vectors of $\R^{\Leaves{T}}$ and define a linear transformation $\Sum:\R^{\Leaves{T}}\to \R^T$ via
\[\Sum: e_w\in \R^{\Leaves{T}}\mapsto \sum_{x\in {\rm Path}_T(w)}\,\sqrt{\frac{\pi(w)}{\pi(x)}}\,e_x.\]
We also abuse notation and denote by $\Sum$ the tensor product of $\Sum$ with the identity operator on $a\in A$. Simple inspection reveals:
\[M_n = \Sum\,\Leaves{M}_n.\]
Therefore, understanding $\Leaves{M}_n$ will lead to an understanding of $M_n$. In particular, the quadratic variations $V_n$ of $M_n$ and $\Leaves{V}_n$ of $\Leaves{M}_n$ are related by:
\[V_n = \Sum\,\Leaves{V}_n\,\Sum'.\]

We will need a Lemma on $\Sum$ that will allow us to compare matrices of the above form. Recall that $\|\cdot\|_\infty$ is the entrywise $\ell_\infty$ norm on matrices.

\begin{lemma}\label{lem:Sinitial} Consider two matrices acting over $\R^{\Leaves{T}}\times \R^A$, both of the form
\[Q:=\sum_{w\in \Leaves{T}}\,e_we_w'\otimes Q_w, \ \ \ \mbox{and} \ \ \ \tilde{Q}:=\sum_{w\in \Leaves{T}}\,e_we_w'\otimes \tilde{Q}_w,\]
where $Q_w,\tilde{Q}_w\in \R^{A\times A}$. Then we have \[\|\Sum\,(Q - \tilde{Q})\,\Sum'\|\leq \max_{w\in \Leaves{T}}\|Q_w - \tilde{Q}_w\|.\]\end{lemma}


Next we shall compute the terms related to the quadratic variations of $\Leaves{M}_n$ and $M_n$. We will also argue that the respective quadratic variations $\Leaves{V}_n$ and $V_n$ are both close to deterministic matrices. We use $\Leaves{\xi}_t = \Leaves{M}_t - \Leaves{M}_{t-1}$ and $\xi_t = M_t-M_{t-1}$ to denote the increments of the two martingales, noting that $\xi_t = \Sum\,\Leaves{\xi}_t$. The first fact we need is this.

\begin{proposition}\label{prop:Sigma}Let $\Sigma_t:=\Exp{t-1}{\xi_t\xi_t'}$. Let $p_t := (p(a|X^{t-1}_{-\infty}))_{a\in A}$ and let $\sqrt{p_t}$ denote the coordinatewise square root of this vector. Then:
\[\Sigma^{1/2}_t = \left(\sum_{w\in \Leaves{T}}\Ind{\{T(X^{t-1}_{-\infty})=w\}}\,\frac{\Sum\,e_we_w'}{\sqrt{\pi(w)\,n}}\right)\otimes [(I-p_t{\bf 1}')\,{\rm diag}(\sqrt{p_t})].\]
Moreover,
\[\Ex{\|\xi_t\|^3_\infty} \leq \frac{C}{n^{3/2}}\sum_{w\in \Leaves{T}}\pi^{-1/2}(w),\]
and if $N\in\R^T\otimes \R^A$ is standard Gaussian and independent of $\sF_{t-1}$,
\[ \Ex{\|\Sigma^{1/2}_tN\|_{\infty}^3} \leq \frac{C}{n^{3/2}}\,\sum_{w\in \Leaves{T}}\pi^{-1/2}(w).\]
\end{proposition}

Proposition \ref{prop:Sigma} controls the impact of the higher order moments that is needed in the application of Theorem \ref{thm:clt:max:martingale}. Provided the choice of tree $T$ has leaves that are not unlikely to be observed, it states that such higher order terms are negligible.

Next we construct deterministic matrices $V$ and $V_\delta$ to approximate $V_n$. We first define, for each $w\in \Leaves{T}$, a (deterministic) matrix $C_w\in \R^{A\times A}$ given by:
\begin{equation}\label{def:Cw}C_w(a,a'):= \left\{\begin{array}{ll}p(a|w)(1-p(a|w)),&  a=a'\in A;\\ - p(a|w)\,p(a'|w), & a\neq a'; a,a'\in A.\end{array}\right.\end{equation}
The matrix defined in (\ref{def:Cw}) is used in the construction of $V$ and $V_\delta$ as follows
$$V:=\Sum\,\Leaves{V}\,\Sum' \ \ \mbox{and} \ \ V_\delta = \Sum\, \Leaves{V}_\delta\,\Sum'$$
where \begin{equation}\label{eq:defVstar}\Leaves{V} :=\sum_{w\in \Leaves{T}}e_we_w \otimes C_w,\end{equation}
and
\begin{equation}\label{eq:defVdeltastar}\Leaves{V}_\delta:=\sum_{w\in \Leaves{T}} e_we_w \otimes (\eps C_w + (1+\eps)\,(2\sqrt{|A|}+1)\,\gamma\,I_{A\times A}).\end{equation}

The following proposition stated the guarantees based on the continuity and typically assumptions.

\begin{proposition}\label{prop:Vbound} Whenever the continuity assumption and the typicality event hold, we have
\[\|V - V_n\|\leq 2\gamma \ \ \mbox{and} \ \ V_n\preceq V+ V_\delta\]
Therefore we have
\[\Pr{V_n\preceq V+ V_\delta} \geq 1-\alpha_\eps, \, \|V_\delta\| \leq \epsilon + (1+\eps)\,(2\sqrt{|A|}+1)\,\gamma,\mbox{ and } \Exp{}{\|V - V_n\|_\infty}\leq 2\gamma+\alpha_\eps\]
where the norm is the entrywise maximum.
\end{proposition}

\bibliographystyle{plain}
\bibliography{martingaleCLT}

\appendix

\section{Proof of the Lindeberg Theorem and its corollary}\label{sec:proof.lindeberg}

We prove our CLT for martingales, \thmref{lindeberg}, along with Corollary \ref{cor:bootstrap}. As a first step, we present in Section \ref{sec:lindebergspecial} the special case of the Theorem where the quadratic variation at time $n$ is deterministic. This proof follows an argument by Lalley. We then present in Section \ref{sec:Gaussian} a Gaussian perturbation lemma that we use to prove the general statement of \thmref{lindeberg} and the Corollary. We take the notation from Section \ref{SEC:NOTATION} for granted and recall that Taylor's formula implies that for all $x,s\in\R^d$,
\begin{equation}\label{taylor}
\frac{s}{\|s\|}\in\sS\Rightarrow \left|\varphi(x + s)  - \varphi(x) - \langle \nabla \varphi(x),s\rangle - \frac{1}{2}\tr(\nabla^2\varphi(x)\,ss')\right|\leq c_3\|s\|^3.\end{equation}

\subsection{A special case}\label{sec:lindebergspecial}

\begin{lemma}[Multidimensional Lindeberg Theorem, special case]\label{lem:CLT:martingale}Under the notation in Section \ref{SEC:NOTATION}, assume our martingale $(M_t,\sF_t)_{t=0}^n$ has quadratic variation $V_n=V$ at time $n$, with $V$ deterministic. Then for all functions $\varphi:\R^d\to \R$ that are three times differentiable, \[\left|\Ex{\varphi\left(M_n\right)} - \Ex{\varphi(N(0,V))}\right|\leq c_3 \sum_{t=1}^n\Ex{\|\xi_t\|^3 + \|\eta_t\|^3},\]
where we recall
\[c_3:= \sup_{x\in\R^d,\,w\in \sS}|\nabla^{3}\varphi(x)(w^{\otimes 3})|.\]\end{lemma}
\begin{proof} We follow the argument for the one-dimensional case by Lalley \cite{lalley2014martingale}.
Consider the intermediate terms:
\[X_k:=\sum_{t=1}^k\,\xi_t + \sum_{t=k+1}^n\,\eta_t,\, k=0,1,2,\dots,k\]
and the leave-one-out variants:
\[X^o_k:=\sum_{t=1}^{k-1}\,\xi_t + \sum_{t=k+1}^n\,\eta_t,\, k=1,2,\dots,k\]
Note that for each $1\leq k\leq n$,
\[X_k = X^o_k + \xi_k\mbox{ and }X_{k-1} = X^o_k + \eta_k.\]
Therefore,
\begin{eqnarray*}\left|\Ex{\varphi\left(M_n\right)} - \Ex{\varphi(N(0,V))}\right|&=& \left|\Ex{\varphi\left(\sum_{t=1}^n\xi_t\right)} - \Ex{\varphi\left(\sum_{t=1}^n\eta_t\right)}\right| \\
&=& \left|\Ex{\varphi(X_n)} - \Ex{\varphi(X_0)}\right| \\&=& \left|\sum_{i=1}^n\Ex{\varphi(X_k)} - \Ex{\varphi(X_{k-1})}\right|\\
&\leq &\sum_{i=1}^n\left|\Ex{\varphi(X^o_k+\xi_k)} - \Ex{\varphi(X^o_{k}+\eta_k)}\right|.\end{eqnarray*}
We will finish the proof by bounding each term in the above sum as follows:
\[\mbox{{\bf Claim:} }\forall k\in[n]\,:\,\left|\Ex{\varphi(X^o_k+\xi_k)} - \Ex{\varphi(X^o_{k}+\eta_k)}\right|\leq c_3\,\Ex{\|\xi_k\|^3 + \|\eta_k\|^3}.\]
To prove this, we fix a $k\in [n]$ from now on. Recall that $\xi_k/\|\xi_k\|,\eta_k/\|\eta_k\|\in \sS$ and use (\ref{taylor}) to deduce:
\begin{eqnarray}\nonumber \Ex{\varphi(X^o_k+\xi_k)} - \Ex{\varphi(X^o_{k}+\eta_k)}&= &\Ex{\langle\nabla\varphi(X^o_k),\xi_k - \eta_k\rangle} \\ \nonumber & & + \frac{1}{2}\Ex{\tr[\nabla^2\varphi(X^o_k)\,(\xi_k\xi_k' - \eta_k\eta_k')]}\\ \label{eq:twoterms}& & + r_k\mbox{ with } |r_k|\leq c_3\,\Ex{\|\eta_k\|^3 + \|\xi_k\|^3}.\end{eqnarray}
If we can show the first two terms in the RHS are zero, we will have obtained the claim. For this we will need to consider the joint distribution of $X_o^k, N_k$ and $\xi_k$. Define the $\sigma$-field \[\sG_{k,n}:=\sigma(\sF_n\vee \sigma(N_k)).\] Note that $N_k$, the matrices $\Sigma_i$ and the random variables $\xi_t$ are all $\sG_{k,n}$-measurable, whereas $N_{k+1},\dots,N_n$ are independent from $\sG_{k,n}$. One consequence of this is that, conditionally on $\sG_{k,n}$
\[X^o_k - \sum_{i=1}^{k-1}\xi_i = \sum_{t=k+1}^n\,\eta_t = \sum_{t=k+1}^n\,\Sigma_i^{1/2}\,N_i =_d N\left(0,\sum_{t=k+1}^n\,\Sigma_i\right) =_d N(0,V-V_k).\]
In particular, letting $\gamma$ denote the standard Gaussian measure over $\R^d$,
\begin{eqnarray*}\nonumber \Ex{\langle\nabla\varphi(X^o_k),\xi_k - \eta_k\rangle\mid\sG_{k,n}} &=& \langle\Ex{\nabla\varphi(X^o_k)\mid\sG_{k,n}},\xi_k - \eta_k\rangle\\
\nonumber &=&\langle h\left(\sum_{t=1}^{k-1}\xi_t ,V-V_k\right),\xi_k - \eta_k\rangle,\\
\nonumber\mbox{where } h(x,M)&:=& \int_{\R^d} \nabla\varphi\left(x + M^{1/2}x'\right)\,\gamma(dx').\end{eqnarray*}
Now, $h\left(\sum_{t=1}^{k-1}\xi_t ,V-V_k\right)$ is $\sF_{k-1}$-measurable, because it is a deterministic function of $\sF_{k-1}$-measurable objects $\sum_{t=1}^{k-1}\xi_t$ and $V-V_k$. We also have $\Exp{k-1}{\xi_k}=\Exp{k-1}{\eta_k}=0$. This means that \begin{eqnarray*}\Ex{\langle\nabla\varphi(W^o_k),\xi_k - \eta_k\rangle} &=&\Ex{\Ex{\langle\nabla\varphi(W^o_k),\xi_k - \eta_k\rangle\mid\sG_{k,n}}}\\ &=&
\Ex{\langle h\left(\sum_{t=1}^{k-1}\xi_t ,V-V_k\right),\xi_k - \eta_k\rangle}\\ &=&\Ex{\langle h\left(\sum_{t=1}^{k-1}\xi_t ,V-V_k\right),\Exp{k-1}{\xi_k - \eta_k}\rangle }= 0.\end{eqnarray*}
This takes care of the term containing the gradient in the RHS of \eqnref{twoterms}. We can use similar reasoning for the Hessian term:
\begin{eqnarray*}\Ex{\tr[\nabla^2\varphi(X_k^o)\,(\xi_k\xi_k' - \eta_k\eta_k')]\mid \sG_{k,n}} &=&\tr[\Ex{\nabla^2\varphi(X_k^o)\mid\sG_n}\,(\xi_k\xi_k' - \eta_k\eta_k')] \\ &=& \tr\left[K\left(\sum_{t=1}^{k-1}\xi_t,V-V_k\right)\,(\xi_k\xi_k' - \eta_k\eta_k')\right],\\
\mbox{where }K(x,M)&:=&\int_{\R^d} \nabla^2\varphi\left(x + M^{1/2}x'\right)\,\gamma(dx').\end{eqnarray*}
Since $K\left(\sum_{t=1}^{k-1}\xi_t,V-V_k\right)$ is $\sF_{k-1}$-measurable and $\Exp{k-1}{\xi_k\xi_k'} = \Exp{k-1}{\eta_k\eta_k'}$,
\begin{eqnarray*}\Ex{\tr\left[\nabla^2\varphi(X_k^o)\,(\xi_k\xi_k' - \eta_k\eta_k')\right]} &=& \Ex{\tr\left[K\left(\sum_{t=1}^{k-1}\xi_t,V-V_k\right)\,(\xi_k\xi_k' - \eta_k\eta_k')\right]} \\&=&
\Ex{\tr\left[K\left(\sum_{t=1}^{k-1}\xi_t,V-V_k\right)\,\Exp{k-1}{(\xi_k\xi_k' - \eta_k\eta_k')}\right]}\\&=&0.\end{eqnarray*}
We conclude that the first two terms in the RHS of \eqnref{twoterms} are $0$. This finishes the proof of the Claim and of the Theorem. \end{proof}

\subsection{Gaussian perturbation}\label{sec:Gaussian}

The next result will be needed in the remainder of the proof.

\begin{lemma}[Gaussian Perturbation Lemma]\label{lem:perturbV}Suppose $V,W\succeq 0$ are $d\times d$ symmetric matrices. Given a $C^3$ function $\varphi:\R^d\to\R$ with bounded second and third derivatives,
\[|\Ex{\varphi(N(0,V))} - \Ex{\varphi(N(0,W))}|\leq \frac{1}{2}\|V-W\|\,\left\|\int_0^1\,\Ex{\nabla^2\varphi(N(0,tW + (1-t)V))}\,dt\right\|_*.\]\end{lemma}
\begin{proof} In fact the proof gives the stronger identity
\[\Ex{\varphi(N(0,V))} - \Ex{\varphi(N(0,W))} = \frac{1}{2}\tr\left((V-W)\,\int_0^1\,\Ex{\nabla^2\varphi(N(0,tW + (1-t)V))}\,dt\right),\]
from which the Lemma follows via the definition of the dual norm.
To prove the identity, we will use an interpolation argument similar to the one in the proof of Lindeberg's theorem.

Let $N_1,\dots,N_m\in \R^d$ be i.i.d. random vectors with $N(0,I)$ distribution, with $m\in\N$ given. Define
\[X_0:= \frac{1}{\sqrt{m}}V^{1/2}\,\sum_{i=1}^mN_i,\]
\[X_k:=\frac{1}{\sqrt{m}}\,\left(W^{1/2}\sum_{i=1}^{k}N_i + V^{1/2}\sum_{i=k+1}^mN_i\right)\,\,(1\leq k\leq m),\]
and the leave-one-out variant \[X_k^o:= \frac{1}{\sqrt{m}}\,\left(W^{1/2}\sum_{i=1}^{k-1}N_i + V^{1/2}\sum_{i=k+1}^mN_i\right)\,\,(1\leq k\leq m).\]
We note $X_0=_d N(0,V)$ and $X_m=_dN(0,W)$ and:
\begin{eqnarray*}\Ex{\varphi(N(0,V))} - \Ex{\varphi(N(0,W))}&=&\Ex{\varphi(X_m) - \varphi(X_0)}\\ &=& \sum_{k=1}^m\Ex{\varphi(X_k) - \varphi(X_{k-1})}.\end{eqnarray*}
We note that $X_k^o$ is independent of $N_k$ and therefore
\begin{eqnarray*}\Ex{\varphi(X_k)} &=& \Ex{\varphi\left(X_k^o+\frac{W^{1/2}}{\sqrt{m}}N_k\right)} \\ &=& \Ex{\varphi\left(X_k^o\right)} +\Ex{\left\langle \nabla \varphi(X_k^o),\frac{W^{1/2}}{\sqrt{m}}N_k\right\rangle}\\ & & +  \frac{1}{2m}\Ex{\left\langle W^{1/2}N_k,\nabla^2 \varphi(X_k^o)W^{1/2}N_k\right\rangle} + \bigoh{m^{-3/2}}\\
&=&  \Ex{\varphi\left(X_k^o\right)} + \frac{1}{2m}\,\tr(W\,\Ex{\nabla^2\varphi(X_k^o)}) + \bigoh{m^{-3/2}}\\
&=& \Ex{\varphi\left(X_k^o\right)} + \frac{1}{2m}\,\tr(W\,\Ex{\nabla^2\varphi(X_k)}) + \bigoh{m^{-3/2}}.\end{eqnarray*}
Similarly,
\[\Ex{\varphi(X_{k-1})} = \Ex{\varphi\left(X_k^o+\frac{V^{1/2}}{\sqrt{m}}N_k\right)} = \Ex{\varphi\left(X_k^o\right)} + \frac{1}{2m}\,\tr(V\,\Ex{\nabla^2\varphi(X_k)}) + \bigoh{m^{-3/2}}.\]
Using the fact that $X_k=_d N(0,tW + (1-t)V)$ with $t=k/m$, we obtain:
\[\Ex{\varphi(X_k) - \varphi(X_{k-1})} = \frac{1}{2m}\tr((V-W)\,\Ex{\nabla^2\varphi(N(0,tW + (1-t)V))}) \mid_{t = \frac{k}{m}} + \bigoh{m^{-3/2}}.\]
We deduce:
\begin{eqnarray*}\Ex{\varphi(N(0,V))} - \Ex{\varphi(N(0,W))} &=& \left(\frac{1}{2m}\sum_{k=1}^m\Ex{\nabla^2\varphi(N(0,tW + (1-t)V))} \mid_{t = \frac{k}{m}}\right) \\ & &+ \bigoh{m^{-1/2}}.\end{eqnarray*}
The first term in the RHS of the above display is a Riemann sum, and the second is small when $m$ is large. Letting $m\to+\infty$ we obtain:
\[\Ex{\varphi(N(0,V))} - \Ex{\varphi(N(0,W))} = \frac{1}{2}\tr\left((V-W)\,\int_0^1\,\Ex{\nabla^2\varphi(N(0,tW + (1-t)V))}\,dt\right).\]\end{proof}

\subsection{The full Theorem \ref{thm:lindeberg}}\label{sec:lindebergfull}

\begin{proof} (of \thmref{lindeberg}) In the first step of our proof we will ``transform"~our martingale into one for which the quadratic variation at time $n+1$ is exactly equal to $V+ V_\delta$. This we accomplish by stopping $M_n$ to avoid that its variance overshoots and, if necessary, adding some noise back to avoid undershooting. The upshot of this transformation is that we can then apply the strategy of Lemma \ref{lem:CLT:martingale}. The change in quadratic variation will be addressed via the Gaussian Perturbation Lemma (Lemma \ref{lem:perturbV} above).

We will require i.i.d. standard normal random vectors $N_1,\dots,N_{n+1}\in\R^d$ defined on the same probability space as the martingale and independent from it. This is as in Section \ref{SEC:NOTATION}, except that we ask for one more vector $N_{n+1}$.

Define the $\{\sF_t\}_{t=0}^n$-stopping time:
\[\tau:=\inf\{m\leq n\,:\, \mbox{either $m=n$ or } V_{m+1}\not\preceq V +  V_\delta\}.\]
Note that $\tau$ is a stopping time because $V_{m+1}$ is always $\sF_m$-measurable. One can check that \[\tilde{M}_{k}:=M_{k\wedge \tau} = \sum_{t=1}^n\tilde{\xi}_t\] where $\tilde{\xi}_k:=\xi_k\Ind{\{\tau\geq k\}}$. Because $\{\tau\geq k\}\in\sF_{k-1}$, we can deduce that $\tilde{M}_k$ is also a martingale. Moreover,
\[\tilde{\Sigma}_k:=\Exp{k-1}{\tilde{\xi}_k\tilde{\xi}_k'} =  \Sigma_k\Ind{\{\tau\geq k\}}\] and the quadratic variation process satisfies
\[\tilde{V}_n = V_{n\wedge \tau}\preceq V +  V_\delta.\]
Notice that $M_n$ and $\tilde{M}_n$ are equal with probability $\geq 1-\alpha$:
\[\Pr{M_n\neq \tilde{M_n}} \leq \Pr{\tau<n} = \Pr{V_n\not\preceq V +  V_\delta} = \alpha.\]
This last expression means that
\begin{equation}\label{eq:firstdiffLindeberg}\left|\Ex{\varphi\left(M_n\right)} - \Ex{\varphi\left(\tilde{M}_n\right)}\right|\leq \alpha\,c_0.\end{equation}
We also see that $\tilde{M}_n$ never overshoots the target quadratic variation $V+ V_\delta$. It may, however, undershoot it. This we fix by defining
\[\tilde{\xi}_{n+1}:=\,(V+ V_\delta-\tilde{V}_n)^{1/2}N_{n+1}\] and considering $\tilde{M}_{n+1} := \tilde{M}_n + \tilde{\xi}_{n+1}$ instead of $\tilde{M}_n$. This is still a martingale and, because
\[\tilde{\Sigma}_{n+1} := \Exp{n}{\tilde{\xi}_{n+1}\tilde{\xi}_{n+1}'} = V+  V_\delta - \tilde{V}_n,\] The quadratic variation of $\tilde{M}_{n+1}$ is exactly equal to $V + V_\delta$. \lemref{perturbV} above, applied conditionally on $\sF_{n}$ (and with $W=0$), allows us to compare the distributions of $\tilde{M}_n$ and $\tilde{M}_{n+1}$:
\begin{eqnarray}\nonumber |\Ex{\varphi(\tilde{M}_n) - \varphi(\tilde{M}_{n+1})\mid\sF_n}|&\leq& (c_2\,\|V+ V_\delta - \tilde{V}_n\|)\wedge (2c_0)\\ \label{eq:seconddifLindeberg}\mbox{($\tilde{V}_n = V_n$ when $V_n\preceq V+ V_\delta$)}&\leq& (2c_0) \Ind{\{V_n\not\preceq V + V_\delta\}} + c_2 \| V_\delta\| + c_2\,\|V - {V}_n\|.\end{eqnarray}

We are now at a stage where we can apply the techniques of Lemma \ref{lem:CLT:martingale} (Lindeberg Theorem for martingales). Introduce \[\tilde{\eta}_k:=\left\{\begin{array}{ll} \Sigma_k^{1/2}\,N_k\Ind{\{\tau\geq k\}},&1\leq k\leq n; \\
 \\
 (V+ V_\delta-\tilde{V}_n)^{1/2}N_{n+1}, & k=n+1.\end{array}\right.\] This choice guarantees that
\[\forall 1\leq k\leq n+1 \,:\,\Exp{k-1}{\tilde{\eta}_k} = \Exp{k-1}{\tilde{\xi}_k}=0\mbox{ and }\Exp{k-1}{\tilde{\eta}_k\tilde{\eta}_k'} = \Exp{k-1}{\tilde{\xi}_k\tilde{\xi}_{k}'}.\]

Define $\tilde{X}_k:=\sum_{t=1}^k\,\tilde{\xi}_t + \sum_{t=k+1}^{n+1}\tilde{\eta}_k$. We wish to compare $\tilde{X}_{n+1}=\tilde{M}_{n+1}$ to $\tilde{X}_0=_dN(0,V+ V_\delta)$. Importantly, even though $\tilde{M}_{n+1}$ is a martingale with $n+1$ time steps, we always have $\tilde{\xi}_{n+1} = \tilde{\eta}_{n+1} = (V+ V_\delta - \tilde{V}_n)^{1/2}N_{n+1}$, so $\tilde{M}_{n+1}=\tilde{X}_{n+1} = \tilde{X}_n$. We conclude:
\[\left|\Ex{\varphi(\tilde{M}_{n+1}) - \varphi(N(0,V+ V_\delta))}\right| = \left|\Ex{\varphi(\tilde{X}_n) - \varphi(\tilde{X}_0)}\right|\leq \sum_{k=1}^n \left|\Ex{\varphi(\tilde{X}_k) - \varphi(\tilde{X}_{k-1})}\right|.\]
We emphasize that $\tilde{M}_{n+1}$ is a martingale whose quadratic variation {\em always} equals $V+ V_\delta$. This means we can proceed as in the Claim in the proof of the first Lindeberg Theorem and bound:
\[ \left|\Ex{\varphi(\tilde{X}_k) - \varphi(\tilde{X}_{k-1})}\right|\leq c_3\,(\Ex{\|\tilde{\xi}_k\|^3} + \Ex{\|\tilde{\eta}_k\|^3}).\]
Now recall that $\|\tilde{\xi}_k\|\leq \|\xi_k\|$ and $|\tilde{\eta}_k|\leq |\eta_k|$ always, and deduce
\begin{equation}\label{eq:thirddifLindeberg}|\Ex{\varphi(\tilde{M}_{n+1}) - \varphi(N(0,V + V_\delta))}|\leq c_3\,\sum_{k=1}^n(\Ex{\|\xi_k\|^3} + \Ex{\|\eta_k\|^3}).\end{equation}
Finally, another application of \lemref{perturbV} gives:
\[|\Ex{\varphi(N(0,V+ V_\delta)) - \varphi(N(0,V))}|\leq c_2\,\|V_\delta\|.\]
The proof follows once we combine the previous display with \eqnref{firstdiffLindeberg}, \eqnref{seconddifLindeberg} and \eqnref{thirddifLindeberg}.\end{proof}

\subsection{Bootstrap as a corollary}\label{sec:proof.bootstrap}

\begin{proof} (of Corollary \ref{cor:bootstrap}) This corollary follows directly from Theorem \ref{thm:lindeberg} and the Gaussian Perturbation Lemma \ref{lem:perturbV}.\end{proof}

\section{Gaussian approximation of the maximum}\label{sec:proof.maximum}

We prove below our Lindeberg result for the maximum (Theorem \ref{thm:clt:max:martingale}; Section \ref{sec:proof:clt:max}) and the anti-concentration result (Corollary \ref{cor:anti-concentration}; Section \ref{sec:proof:anti-concentration}).

\subsection{Proof of Theorem \ref{thm:clt:max:martingale}}\label{sec:proof:clt:max}

We prove here our Theorem for approximating the distribution of the maximum.

\begin{proof}
The proof follows builds upon the LSE approximation for the max used in the proof of Lemma 5.1 in \cite{chernozhukov2014clt}, Theorem \ref{thm:lindeberg} for martingales and the bootstrap Corollary \ref{cor:bootstrap}. For simplicity, we omit the argument for the law of $\widetilde Z_n$, as it follows directly from the Corollary.

For a Borel set $A \subset \mathbb{R}$, define its $\epsilon$-enlargement as $A^\epsilon=\{ t \in \mathbb{R} : {\rm dist}(t,A) \leq \epsilon\}$.  Define $\mu_{tj} = \Exp{t-1}{\xi_{tj}}$ (note that in our setting $\mu_{tj}=0$), $\bar \mu_j = \sum_{t=1}^n \mu_{tj}$ and $\bar\mu = (\mu_{j})_{j=1}^p$. For a vector $v \in \mathbb{R}^d$ we let $$F_{\beta}(v+\bar \mu)=F_{\beta,\bar \mu}(v)=\beta^{-1}\log\left(\sum_{j=1}^p\exp(\beta \{v_j+\bar \mu_{j}\})\right)$$

It follows by the definition of the LSE approximation of the maximum that
\begin{equation*}\label{LSE:prop} 0 \leq F_\beta(v) - \max_{1\leq j \leq d} v_{j} \leq \beta^{-1} \log d, \ \ \mbox{for all} \ v \in \mathbb{R}^{d}. \end{equation*}
By Lemma 5.1 in \cite{chernozhukov2015noncenteredprocesses}, for each Borel set $A\subset \mathbb{R}$ and $\bar\delta>0$, there exists a function $g \in C^3$, satisfying $\|g'\|_\infty \leq \bar\delta^{-1}$, $\|g''\|_\infty \leq \bar\delta^{-2}K$, $\|g'''\|_\infty\leq \bar\delta^{-3}K$ for a universal constant $K$, such that
$ 1_{A^{\bar\delta}}(t) \leq g(t) \leq 1_{A^{4\bar\delta}}(t)$ for all $t\in \mathbb{R}$.
We will choose $\beta$ so that $\bar \delta = \beta^{-1}\log (d)$, so that
\begin{equation}\label{eq:LSEprop2}0 \leq F_\beta(v) - \max_{1\leq j \leq d} v_{j} \leq \bar\delta, \ \ \mbox{for all} \ v \in \mathbb{R}^{d}.\end{equation}

Our Lindeberg machinery will be applied to the function:
\begin{equation}\label{eq:defphimaximum}\varphi:=g\circ F_\beta,\end{equation}
which (by virtue of (\ref{eq:LSEprop2})) satisfies:
\[\forall v\in \R^d\,:\,1_A(v)\leq 1_{A^{\bar\delta}}(F_\beta(v))\leq \varphi(v)\leq 1_{A^{4\bar\delta}}(F_\beta(v))\leq 1_{A^{5\bar\delta}}(v).\]

Therefore,
\begin{equation*}\Pr{Z\in A} - \Pr{\widetilde Z\in A^{5\bar\delta}}\leq \Ex{\varphi(M_n)} - \Ex{\varphi(H)}\end{equation*}
where $H$ has distribution $N(0,V)$. We conclude from Theorem \ref{thm:lindeberg} that:
\begin{equation}\label{eq:maxboundlindeberg}\Pr{Z\in A} - \Pr{\widetilde Z\in A^{5\bar\delta}}\leq c_0\alpha + 2c_2\Delta_n + c_3 \sum_{t=1}^n\Ex{\|\xi_t\|^3_\infty + \|\eta_t\|^3_\infty},\end{equation}
where the $c_k$ are defned right before the Theorem. Crucially, we have chosen to apply the $\ell_\infty$ norm on vectors in what follows.

It remains to bound $c_0$, $c_2$ and $c_3$. Since $0\leq g\leq 1$, we can bound $c_0\leq 2$. The other values require that we compute derivatives of $\varphi$, as
\[c_k:=\sup_{x\in\R^d,w\in\R^d,\|w\|_\infty\leq 1}|\nabla^k\varphi(x)(w^{\otimes k})|.\]
In fact, we {\em claim} that we can bound:
\[\mbox{\bf Claim: }c_2\leq C\,\frac{\log d}{\bar\delta^2} \mbox{ and }c_3\leq C\,\frac{\log^2d}{\bar\delta^3}.\]
This suffices to finish the proof by a direct plugin into (\ref{eq:maxboundlindeberg}).

To prove the claim, we will use the simple bounds:
\[c_2\leq \sup_{v\in\R^d}\sum_{j,k=1}^d|\partial_j\partial_k\varphi(v)| \mbox{ and }c_3\leq \sup_{v\in\R^d}\sum_{j,k,\ell=1}^d|\partial_j\partial_k\partial_\ell\varphi(v)|.\]
We also use some formulae from  \cite[Lemmas A.2 and A.4]{chernozhukov2013gaussian}. For indices $1\leq i,j,k,\ell\leq d$ and $v\in\R^d$, if we define
\begin{eqnarray*}\delta_{jk}&:=& 1_{\{j=k\}};\\
\pi_j(v) &:=& \frac{e^{\beta\,v_j}}{\sum_{i=1}^de^{\beta\,v_i}};\\
w_{jk}(v) &:=& \pi_j(v)\delta_{jk} - \pi_j(v)\pi_k(v);\\
q_{jk\ell}(v)&:=& \pi_j(v)\delta_{jk}\delta_{j\ell}  - \pi_j(v)\pi_\ell(v)\delta_{jk} \\ & & - \pi_j(v)\pi_k(v)(\delta_{j\ell}+\delta_{k\ell}) + 2\pi_j(v)\pi_k(v)\pi_\ell(v); \end{eqnarray*}
then (in our notation):
\begin{eqnarray*}
\partial_j\partial_k\varphi(v)&=&g''(F_\beta(v))\,\pi_j(v)\pi_k(v) + g'(F_\beta(v))\,\beta\,w_{jk}(v);\\
\partial_j\partial_k\partial_\ell\varphi(v)&=&g'''(F_\beta(v))\,\pi_j(v)\pi_k(v)\pi_\ell(v)  +\beta^2\,g'(F_\beta(v))\,q_{jk\ell}(v) \\ & & + \beta\,g''(F_\beta(v))\,(w_{jk}(v)\pi_\ell(v) + w_{j\ell}(v)\pi_k(v) + w_{k\ell}(v)\pi_j(v)).\end{eqnarray*}
The important thing about these formulae is that, for each fixed $v\in\R^d$, $\pi_j(v)$ is a probability vector and has $\ell_1$ norm equal to $1$. Therefore,
\[\sum_{j,k=1}^d|w_{jk}(v)|\leq C_0\mbox{ and }\sum_{j,k,\ell=1}^d|q_{jk\ell}(v)|\leq C_0\]
with $C_0$ independent of $v$, $\bar\delta$ or any other parameter of the problem. We may then prove the claim combining these estimates with the formulae for the partial derivatives of $\varphi$, our bounds on the derivatives of $g$, and the fact that $\beta=\log d/\bar\delta$.\end{proof}

\subsection{Proof of Anti-concentration}\label{sec:proof:anti-concentration}

\begin{proof} (of Theorem \ref{thm:max:anticoncentration})
Let $\bar \sigma_0 = \bar \sigma \log^{-1/2} d$, and $I_0 = \{ j \in [d] : \sigma_j \in (\bar\sigma_0,\bar\sigma]\}$. Since $X_j$ are centered Gaussian random variables, the density function of $X_j$ at $t\geq 0$, $f(t,\sigma_j)=(2\pi\sigma_j^2)^{-1/2}\exp(-t^2/\{2\sigma_j^2\})$, is decreasing in $t$. Moreover, $f(t,\sigma_j)$ is increasing in $\sigma_j$ if $|t|\geq \sigma_j$. Finally, for $t\geq \frac{4}{3}\bar\sigma$, $\epsilon \leq t/4$, we have $t-\epsilon\geq\frac{7}{8}t\geq \frac{7}{8}\Phi^{-1}(0.95)\geq\sqrt{2}\bar\sigma$. Therefore it follows that
%
$$\begin{array}{rl}
 \Pr{ |\max_{j\in [d]\setminus I_0}  X_j  - t | \leq \epsilon } & \leq \sum_{j \in [d]\setminus I_0} \Pr{ |  X_j  - t| \leq \epsilon } \\
 & \leq 2\epsilon \sum_{j \in [d]\setminus I_0} \exp{(-(t-\epsilon)^2/\{2\bar\sigma_j^2\})}/\{\sqrt{2\pi}\bar\sigma_{j}\} \\
 & \leq 2\epsilon \{d-|I_0|\} \exp{(- 2\bar\sigma^2/\{2\bar\sigma_{0}^2\})}/\{\sqrt{2\pi}\bar\sigma_{0}\}\\
 & \leq \{\epsilon/\bar\sigma_0\} \exp{(- \log d+\log (d-|I_0|))} \\
 & \leq \{\epsilon/\bar\sigma_0\}\end{array}
 $$
Therefore, we have that
$$\begin{array}{rl}
 \Pr{ |\max_{j\in[d]} X_j - t |\leq \epsilon } & \leq \Pr{ |\max_{j\in [d]\setminus I_0} X_j - t |\leq \epsilon }\\
  & + \Pr{ |\max_{j\in I_0} X_j - t |\leq \epsilon }\\
 & \leq \{\epsilon/\bar\sigma_0\} + \{2\epsilon/\bar\sigma_0\} \{2\sqrt{\log (2|I_0|)}+2\} \\
 \end{array}
 $$
 where the bound on the second term follows from Lemma 4.3 in \cite{chernozhukov2015noncenteredprocesses} since for each $j\in I_0$ we have $\sigma_j \geq\bar\sigma \log^{-1/2}d$ and (trivially) $|I_0|\leq d$.
\end{proof}

\section{Proofs of Section \ref{SEC:MM}}

\begin{proof} (of Theorem \ref{thm:MMinfty})
Let
$$\begin{array}{rl}
 Z = \max_{j\in [d]}|G_j|,\ \ Z^* = \max_{j\in [d]}|\sum_{k=1}^n g_k  \sigma_j^{-1}(r_{kj}-\mu_j)|, \\  \mbox{and} \ \ \widehat Z^* = \max_{j\in [d]}|\sum_{k=1}^n g_k \hat\sigma_j^{-1} (r_{kj}-\hat\mu_j)|\end{array}$$
where $G \in \mathbb{R}^d$ is a zero mean Gaussian vector with the covariance structure equal to $\frac{1}{n}\sum_{k=1}^n\Exp{k-1}{Z_kZ_k'}$. We let $\widehat{{\rm cv}}(\delta)$ denote the conditional quantile of $\widehat Z^*$ given the data and ${{\rm cv}^o}(\delta)$  the $(1-\delta)$-quantile of $Z$.

By (\ref{def:SCB}) it suffices to show
$$ \left| \Pr{ \max_{j\in[d]} \left| \hat\sigma^{-1}_j(\hat\mu_j - \mu)\right| > \widehat{\rm cv}(\delta)  } - (1-\delta) \right|\leq o(\delta). $$

Therefore, we have
$$\begin{array}{rl}
 \displaystyle \Pr{ \max_{j\in[d]} \left| \hat\sigma^{-1}_j(\hat\mu_j - \mu)\right| > \widehat{{\rm cv}}(\delta)  }& \displaystyle \leq_{(1)} \Pr{ \max_{j\in[d]} \left| \sigma^{-1}_j(\hat\mu_j - \mu)\right| > \widehat{{\rm cv}}(\delta) \min_{j\in[d]}|\hat \sigma_j/\sigma_j| } \\
 & \displaystyle \leq_{(2)} \Pr{ \max_{j\in[d]} \left| \sigma^{-1}_j(\hat\mu_j - \mu)\right| > {{\rm cv}}^o(\delta+\vartheta_n) } + o(\delta) \\
 & \displaystyle \leq_{(3)} \Pr{ \max_{j\in[d]} \left| G_j\right| > {{\rm cv}}^o(\delta+\vartheta_n) } + o(\delta) \\
  & \displaystyle \leq_{(4)} \delta + \vartheta_n + o(\delta)
 \end{array}$$
\noindent where  (1) follows by simple arithmetics, (2) from Step 2, and (3) from Step 3 below. (4) follows by definition of ${\rm cv}^o(\delta+\vartheta_n)$.

Step 2. Here we show that for $\vartheta_n = C\delta_n/\log d $ we have $$\Pr{\widehat{{\rm cv}}(\delta) \min_{j\in[d]}|\hat \sigma_j/\sigma_j| >  {{\rm cv}}^o(\delta+\vartheta_n)  } \geq  1-C(\psi_n^{1/2}+n^{-1/2}+\delta\delta_n^{1/4})$$

First note that
$$ \min_{j\in[d]}|\hat \sigma_j/\sigma_j| \geq 1 - \max_{j\in [d]} \left| \frac{\hat\sigma_j-\sigma_j}{\sigma_j} \right| $$

Uniformly over $j\in [d]$, we have
$$ \begin{array}{rl}
|\hat\sigma_j-\sigma_j| & = \left|\left\{\frac{1}{n}\sum_{k=1}^n \hat Z_{kj}^2\right\}^{1/2} - \left\{\frac{1}{n}\sum_{k=1}^n \Exp{k-1}{Z_{kj}^2}\right\}^{1/2}\right| \\
& \leq \left\{\frac{1}{n}\sum_{k=1}^n (\hat Z_{kj}-Z_{kj})^2\right\}^{1/2}\\
& + \left| \left\{\frac{1}{n}\sum_{k=1}^n Z_{kj}^2\right\}^{1/2}- \left\{\frac{1}{n}\sum_{k=1}^n \Exp{k-1}{Z_{kj}^2}\right\}^{1/2}\right|\\
& \leq C\delta \log^{-1}(dn) + \left|\frac{\frac{1}{n}\sum_{k=1}^n Z_{kj}^2-\Exp{k-1}{Z_{kj}^2}}{\left\{\frac{1}{n}\sum_{k=1}^n Z_{kj}^2\right\}^{1/2}+\left\{\frac{1}{n}\sum_{k=1}^n\Exp{k-1}{Z_{kj}^2}\right\}^{1/2}}\right|\\
& \leq C\delta_n \log^{-1}(dn) + \sigma^{-1}_j C\delta_n \log^{-1}(dn)
\end{array} $$
with probability $1-\psi_n$ by  Assumption \ref{assum:MMinfty}. Moreover, we have that $\min_{j\in[d]} \sigma_j \geq c$. Thus, for $\epsilon_n = C\delta_n/\log(dn)$ we have $$\Pr{ \min_{j\in[d]}|\hat\sigma_j/\sigma_j| < 1-\epsilon_n } \leq \psi_n  $$

Next we relate the quantiles ${{\rm cv}}^o(\delta+\vartheta_n)$ of $Z$ and the quantiles $\widehat{{\rm cv}}(\delta)$ of $\widehat Z^*$. Using the triangle inequality we have
$$\begin{array}{rl}
|\widehat Z^* - Z| & \leq |\widehat Z^* - Z^*| + |Z^* - Z| \\
& \leq \max_{j\in [d]} \left|\frac{1}{\sqrt{n}}\sum_{k=1}^n g_k\{\hat Z_{kj}-Z_{kj}\}\right| + |Z^* - Z| \\
& \leq C\log^{1/2}(dn) \max_{j\in [p]} \left\{\frac{1}{n}\sum_{k=1}^n (\hat Z_{kj} - Z_{kj})^2\right\}^{1/2}+ |Z^* - Z|\\
& \leq C\delta_n\log^{-1/2}(dn) + |Z^* - Z|\end{array}$$
with probability $1-\psi_n-2/n$. Indeed, $\frac{1}{\sqrt{n}}\sum_{k=1}^n g_k\{\hat Z_{kj}-Z_{kj}\}$ is a mean zero Gaussian random variable conditionally on the data, where the variance of each component is bounded by $\delta_n^2/\log^2d$ by Assumption \ref{assum:MMinfty} holds with probability $1-\psi_n$. This follows by Corollary 2.2.8 in \cite{van1996weak} and Proposition A.2.1 in \cite{van1996weak} (Borell-Sudakov-Tsirel’son inequality).

To bound the second term we note that both $Z$ and $Z^*$ are the maximum of Gaussian processes with covariance matrices $\Sigma$ and $\Sigma^{Z^*}$ satisfying with probability $1-2\psi_n$ that
$$\begin{array}{rl}
\Delta & :=  \max_{j,\ell \in [d]} |\Sigma^Z_{j\ell} - \Sigma^{Z^*}_{j\ell}| \\
& \leq  \max_{j,\ell \in [d]} | V_{j\ell} - V_{n,j\ell}   | + \max_{j,\ell \in [d]} | V_{nj\ell} - \frac{1}{n}\sum_{k=1}^nZ_{kj}Z_{k\ell}   |\\
& \leq  2\delta_n/\log^2(dn) =: \bar\Delta \end{array}$$
Define the event  $ E_n :=  \{\Delta \leq \bar\Delta\}$. Conditionally on $E_n$, by the perturbation Lemma \ref{lem:perturbV} we have
$$ \Pr{Z^* \in A \mid  E_n } \leq  \Pr{Z  \in A^{\bar\delta} } + C\bar\delta^{-2}\bar\Delta \log(d) $$
Let  $r_n := \bar\delta + C\delta_n\log^{-1/2}(dn)$. By a conditional version of Strassen's theorem, there is a version of $Z$ such that
$$ \wp_n^2 := \Pr{ |\widehat Z^* - Z| > r_n } \leq 2\psi_n + 2/n + C\bar\delta^{-2}\bar\Delta \log(d). $$
By Markov's inequality we have that with probability $1-\wp_n$
$$ \Pr{ |\widehat Z^* - Z| > r_n \mid E_n} \leq \wp_n$$
Using these relations, for some $\vartheta_n \geq \wp_n$, with probability $1-\wp_n$, we have
$$\begin{array}{rl}
\hat {\rm cv}(\delta) (1-\epsilon_n) & \geq_{(1)}  \{{\rm cv}^o(\delta+\wp_n) - r_n \}(1-\epsilon_n)\\
& =_{(2)} {\rm cv}^o(\delta+\vartheta_n) - \epsilon {\rm cv}^o(\delta+\vartheta_n)\\
& + \{{\rm cv}^o(\delta+\wp_n)-{\rm cv}^o(\delta+\vartheta_n) - r_n \}(1-\epsilon_n)\\
& \geq_{(3)} {\rm cv}^o(\delta+\vartheta_n) +  \left\{ \frac{\vartheta_n -\wp_n}{C\log^{1/2} d } - r_n \right\}(1-\epsilon_n)-\epsilon{\rm cv}^o(\delta+\vartheta_n)\\
& \geq_{(4)} {\rm cv}^o(\delta+\vartheta_n) \\
\end{array}
$$
where (1) follows by the definition of the quantile function, (2) by adding and subtracting $(1-\epsilon) {\rm cv}^o(\delta+\vartheta_n)$, (3) by Theorem \ref{thm:max:anticoncentration} and Assumption \ref{assump:C}(ii) which implies $\underline{\sigma} \geq c$, and (4) holds provided that
\begin{equation}\label{eq:rel-anticonSec}   r_n  + \epsilon_n\frac{{\rm cv}^o(\delta+\vartheta_n)}{1-\epsilon_n} \leq \frac{\vartheta_n -\wp_n}{C\log^{1/2} d }\end{equation}
where ${\rm cv}^o(\delta+\vartheta_n) \leq C \log^{1/2}(d/\delta)$. Thus (\ref{eq:rel-anticonSec}) yields
$$\Pr{  \widehat{{\rm cv}}(\delta)(1-\epsilon_n) < {{\rm cv}^o}(\delta+\vartheta_n) } \leq \wp_n $$

Next we define $\bar \delta = \{ c^{-2}\delta^{-2}\delta_n^{-1/2} \bar\Delta \log(d) \}^{1/2}$ and $\vartheta_n = \delta \delta_n^{1/8}$ and show their validity. This implies that $\wp_n^2 \leq 2\psi_n + 2/n + Cc\delta^2 \delta_n^{1/2}$ which in turn implies $\vartheta_n -\wp_n \geq \vartheta_n/2$ since $\psi_n^{1/2} \leq o(1) \delta\delta^{1/8}_n$.
Under these choices, (\ref{eq:rel-anticonSec}) holds provided that
$$\delta_n \leq o(1) \delta\delta_n^{1/8}, \ \  \epsilon_n = \frac{C\delta_n}{\log (dn)} \leq  \frac{o(1)\delta\delta_n^{1/8}}{\log(d/\delta)},\ \ \ \bar \delta  \leq \frac{c^{-1}\delta^{-1}\delta_n^{1/4}}{\log^{1/2}(dn)} \leq \frac{o(1)\delta\delta_n^{1/8}}{\log^{1/2} (d)}$$
which are implied by Assumption \ref{assum:MMinfty} and $n$ sufficiently large as the sequence $\delta_n\to 0$ is fixed.

Step 3. Here we show that
$$ \sup_{t \in \mathbb{R}} \left| \Pr{ \max_{j\in [d]} |G_j| > t } - \Pr{\max_{j\in[d]} \left| \frac{1}{\sqrt{n}}\sum_{k=1}^n Z_{kj} \right| > t } \right| \leq 2\psi_n + C\delta_n^{1/3} $$
where $G_j = \frac{1}{\sqrt{n}}\sum_{k=1}^n \eta_k$ is a Gaussian process with $\eta_k \sim N(0,\Exp{k-1}{Z_kZ_k'})$. By Theorem \ref{thm:clt:max:martingale} we have that
$$ \Pr{\max_{j\in[d]} \left| \frac{1}{\sqrt{n}}\sum_{k=1}^n Z_{kj} \right| > t }  \leq \Pr{ \max_{j\in [d]} |G_j| > t - C\bar\delta } + 2\alpha + \frac{C\Delta_n \log d}{\bar\delta^{2}} + \frac{C\log^2d}{\bar\delta^{3}n^{1/2-\rho}} $$
$$ \ \ \ \ \  \ \ \ \  \ \ \ \ \ \ \ \ \ \ \ \  \ \ \ \leq \Pr{ \max_{j\in [d]} |G_j| > t } +  C\bar\delta\sqrt{\log d} + 2\alpha + \frac{C\Delta_n \log d}{\bar\delta^{2}} + \frac{C\log^2d}{\bar\delta^{3}n^{1/2-\rho}} $$
where $\alpha$ and $\Delta_n$ are defined in Assumption \ref{assump:A}, and the second line used the anti-concentration bound in Corollary \ref{cor:anti-concentration} in the second line ($\underline{\sigma}=1$ since the $Z_{kj}$ are normalized). Note that Assumption \ref{assump:A} holds with $\alpha = \psi_n$, $V_\delta = \{\delta_n/\log^2 (dn)\}V$, and $$\Delta_n \leq \max_{j,\ell\in [d]}|V_{\delta,j\ell}| + \Exp{}{\max_{j,\ell\in[d]}|V_{n,j\ell}-V_{j\ell}|} \leq  \frac{\delta_n \max_{j,\ell}|V_{j\ell}|}{\log^2(dn)} + \frac{\delta_n }{\log^2(dn)} $$by Assumption \ref{assum:MMinfty}.  Moreover, by taking $\bar\delta = \delta_n^{1/3} / \sqrt{\log d}$ we have
$$ \Pr{\max_{j\in[d]} \left| \frac{1}{\sqrt{n}}\sum_{k=1}^n Z_{kj} \right| > t }  \leq \Pr{ \max_{j\in [d]} |G_j| > t } + C\delta_n^{2/3} + 2\psi_n + C\delta_n^{1/3} + C\delta_n $$
under Assumption \ref{assum:MMinfty}  the result follows. The other direction follows similarly.

\end{proof}

\section{Proofs of Section \ref{SECCONTEXT}}

\begin{proof} (of Theorem \ref{thm:ContextFinal})
By triangle inequality we have
$$
 \begin{array}{rl}
 \displaystyle \Pr{\bigcup_{w \in E_n, a\in A}\left\{ \frac{|\bar p(a\mid w) - \hat p(a\mid w)|}{\widehat{{\rm cf}}(w)} > 1  \right\}}& \displaystyle  \leq \Pr{ \max_{w \in T, a\in A}\frac{|\bar p(a\mid w) - \hat p(a\mid w)|}{\widehat{{\rm cf}}(w)} > 1  } \\
 & \displaystyle + \Pr{ \max_{\substack{w \in E_n \setminus T,\\ a\in A}}\frac{|\bar p(a\mid w) - \hat p(a\mid w)|}{\widehat{{\rm cf}}(w)} > 1  }
 \end{array}$$
By the choice of $\widehat {\rm cf}(w)$, $w\not\in T$, in (\ref{eq:old}) with $\delta/n$, it was shown in  Theorem 1 of \cite{belloni2017approximate} that
$$ \Pr{ \max_{\substack{w \in E_n \setminus T,\\ a\in A}}\frac{|\bar p(a\mid w) - \hat p(a\mid w)|}{\widehat{{\rm cf}}(w)} > 1  } \leq \delta/n $$

Next we focus on the other term. Let
$$\begin{array}{rl}
&  Z = \max_{\substack{w \in T,\\ a\in A}}|G_n(w,a)|,\\
&  Z^* = \max_{\substack{w \in T,\\ a\in A}}|\sum_{k=1}^n g_k  d_k(w,a)|, \ \mbox{and}\\
&\widehat Z^* = \max_{\substack{w \in T,\\ a\in A}}|\sum_{k=1}^n g_k  \hat d_k(w,a)|.\end{array}$$

Recall that $\widehat{{\rm cv}}(\delta)$ denotes the conditional quantile of $\widehat Z^*$ given the data $(X^n_{-\infty})$ and ${{\rm cv}^o}(\delta)$ denotes  the $(1-\delta)$-quantile of $Z$.

By definition we have
$$ \frac{|\bar p(a\mid w) - \hat p(a\mid w)|}{\widehat{{\rm cf}}(w)} = |M_n(w,a)|\frac{\sqrt{\pi(w)n}}{N_{n-1}(w)}\frac{\sqrt{N_{n-1}(w)}}{\widehat{{\rm cv}}(\delta)}=\frac{|M_n(w,a)|}{\widehat{{\rm cv}}(\delta)} \sqrt{\frac{\pi(w)n}{N_{n-1}(w)}} $$
Therefore, using the relation above and denoting ${{\rm cv}^o}(\delta)$ as the $1-\delta$ quantile of $Z$, we have
$$\begin{array}{rl}
 & \displaystyle \Pr{ \max_{\substack{w \in T,\\ a\in A}}\frac{|\bar p(a\mid w) - \hat p(a\mid w)|}{\widehat{{\rm cf}}(w)} > 1  } \\
 & \displaystyle  \leq \Pr{ \max_{\substack{w \in T,\\ a\in A}}\frac{|M_n(w,a)|}{\widehat{{\rm cv}}(\delta)} \sqrt{\frac{\pi(w)n}{N_{n-1}(w)}} > 1  } \\
 & \displaystyle \leq_{(1)} \Pr{ \max_{\substack{w \in T,\\ a\in A}}|M_n(w,a)|  > \widehat{{\rm cv}}(\delta)\sqrt{1-\epsilon} } + \alpha_\eps \\
 & \displaystyle \leq_{(2)} \Pr{ \max_{\substack{w \in T,\\ a\in A}}|M_n(w,a)|  > {{\rm cv}^o}(\delta+\vartheta_n) } + \alpha_\eps + C\delta/\log n \\
 & \displaystyle \leq_{(3)} \Pr{ \max_{\substack{w \in T,\\ a\in A}}|G_n(w,a)|  > {{\rm cv}^o}(\delta+\vartheta_n)} + \alpha_\eps + o(1) \\
& \leq_{(4)} \delta + \vartheta_n + \alpha_\eps + o(1) \end{array}$$
where (1) follows from the typicality assumption, (2) follows from Step 2 below where $\vartheta_n = o(\delta)$, (3) by Step 3 below and (4) by definition of the quantile.\\

\underline{Step 2.} We show that $\Pr{  \widehat{{\rm cv}}(\delta)\sqrt{1-\epsilon} < {{\rm cv}^o}(\delta+\vartheta_n) } \leq \delta/\sqrt{n} + \delta/\log n$ for $\vartheta_n = \delta/\log n = o(1)$.

We have that
\begin{equation}\label{eq:triStart}
\begin{array}{rl}
 |\widehat Z^* - Z | & \leq |\widehat Z^* - Z^* | + |Z^* - Z |\\
& = \max_{\substack{w \in T,\\ a\in A}}|\sum_{k=1}^n g_k \{ \hat d_k(w,a)- d_k(w,a)\}| + |Z^* - Z |\\
 \end{array}
\end{equation}
Regarding the first term, note that $\sum_{k=1}^n g_k \{ \hat d_k(w,a)- d_k(w,a)\}$ is a zero-mean Gaussian random variable with variance $\sum_{k=1}^n \{ \hat d_k(w,a)- d_k(w,a)\}^2$. Note that under the typicality, with probability $1-\alpha_\eps$, uniformly over $w\in \Leaves{T}$, $a\in A$ we have
$$\begin{array}{rl}
 | \hat d_k(w,a)- d_k(w,a) | & \leq \left\{\frac{\gamma+|\bar p(a\mid w) - \hat p(a\mid w)|}{\sqrt{\pi(w)n}}+ \left|\sqrt{\frac{N_{n-1}(w)}{\pi(w)n}}-1\right|\frac{1}{\sqrt{N_{n-1}(w)}}\right\}\Ind{\{X^{k-1}_{k-|w|}=w\}}\\
 & \leq \left\{\gamma+|\bar p(a\mid w) - \hat p(a\mid w)|+\epsilon/(1-\epsilon)\right\}\Ind{\{X^{k-1}_{k-|w|}=w\}}/\sqrt{\pi(w)n}\\
\end{array}. $$
However, by typicality, (\ref{eq:old}) and Theorem 1 of \cite{belloni2017approximate} we have that with probability $1-\delta^2/n-\alpha_\eps$
\begin{equation}\label{eq:old2}
\begin{array}{rl}
|\bar p(a\mid w) - \hat p(a\mid w)| & \leq \sqrt{\frac{4}{N_{n-1}(w)}\left(2\log(2+\log_2 N_{n-1}(w))+\log(n^2|A|/\delta)\right) }\\
& \leq \sqrt{\frac{4(1+\epsilon)}{\pi(w)n}\left(2\log(2+\log_2 ((1+\epsilon)\pi(w)n))+\log(n^3|A|/\delta)\right) }\\
& \leq \sqrt{\frac{16}{\pi(w)n}\log(n^3|A|/\delta^2) }\\
\end{array}
\end{equation}

Therefore, with probability $1-2\alpha_\eps-\delta^2/n$, uniformly over $w\in T, a\in A$, we have
\begin{equation}
\begin{array}{rl}
\displaystyle \sum_{k=1}^n \{ \hat d_k(w,a)- d_k(w,a)\}^2& \leq \left\{\gamma +\sqrt{\frac{16}{\pi(w)n}\log(n^3|A|/\delta) } +\frac{\epsilon}{1-\epsilon} \right\}^2 {\displaystyle \sum_{k=1}^n} \frac{\Ind{\{X^{k-1}_{k-|w|}=w\}}}{\pi(w)n} \\
& = \left\{\gamma +\sqrt{\frac{16}{\pi(w)n}\log(n^3|A|/\delta^2) } +\frac{\epsilon}{1-\epsilon} \right\}^2 {\displaystyle  \frac{N_{n-1}(w)}{\pi(w)n}} \\
& \leq (1+\eps)\left\{\gamma +\sqrt{\frac{16}{\pi(w)n}\log(n^3|A|/\delta^2) } +\frac{\epsilon}{1-\epsilon} \right\}^2 =: \varepsilon_n^2(w).
\end{array}
\end{equation}
In turn, by Corollary 2.2.8 in \cite{van1996weak},   with probability $1-2\alpha_\eps-\delta^2/n$ we have
$$ \mathcal{E}_{X^n_{-\infty}}:=\Exp{}{\max_{w\in \Leaves{T},a\in A}\left|\sum_{k=1}^n g_k \{ \hat d_k(w,a)- d_k(w,a)\}\right| \mid X^{n}_{-\infty} } \leq C \sqrt{\log d} \max_{w\in \Leaves{T}}\varepsilon_n(w) $$
and, conditional on the same event with probability at least $1-\alpha_\eps-\delta^2/n$, by Proposition A.2.1 in \cite{van1996weak} (Borell-–Sudakov-–Tsirel'son inequality), we have
$$ \Pr{ \max_{ \substack{w \in T,\\ a\in A}}\left|\sum_{k=1}^n g_k \{ \hat d_k(w,a)- d_k(w,a)\}\right| > \mathcal{E}_{X^n_{-\infty}} +  \max_{w\in \Leaves{T}}\varepsilon_n(w) \sqrt{2\log\left(\frac{n}{\delta^2}\right)} \mid X^{n}_{-\infty} } \leq \frac{2\delta^2}{n}. $$

To bound the other term in (\ref{eq:triStart}) we note that $Z^*$ and $Z$ are each the maximum of a Gaussian random vector so we will apply the (Gaussian) perturbation Lemma. Let
$$ \Delta = \max_{\substack{w,w'\in \Leaves{T}\\ a,a'\in A}} \left| \sum_{k=1}^n d_k(w,a)d_k(w',a') -  \sum_{k=1}^n \Exp{k-1}{d_k(w,a)d_k(w',a')}\right|$$
and $U_k(w,a):= \{\Ind{\{X_k=a\}}-p(a\mid X^{k-1}_{-\infty})\}\Ind{\{X^{k-1}_{k-|w|-1}=w\}}$. We have
$$ \begin{array}{rl}
\Pr{\Delta > t } & \leq \displaystyle d \max_{\substack{w,w'\in \Leaves{T}\\ a,a'\in A}} \Pr{  \left| \sum_{k=1}^n d_k(w,a)d_k(w',a') -  \sum_{k=1}^n \Exp{k-1}{d_k(w,a)d_k(w',a')}\right| > t }\\
&\leq \displaystyle d \max_{\substack{w\in \Leaves{T}\\ a,a'\in A}} \Pr{  \left| \sum_{k=1}^n U_k(w,a)U_k(w,a') - \Exp{k-1}{U_k(w,a)U_k(w,a')} \right| > t n \pi(w) }\\
& \leq 2d\exp(- \frac{1}{4}t^2n\min_{w\in\Leaves{T}}\pi(w))
\end{array}$$
for any $$t \leq \bar \Delta := \frac{2\sqrt{\log(2nd/\delta^2)}}{n^{1/2}\min_{w\in\Leaves{T}}\pi^{1/2}(w)}$$
where the last  step follows from Lemma 1.6 in \cite{LedouxTalagrandBook} (with $b^2 = n\pi(w)$, $a=1$ and $at/b^2\leq \log(3/2)$) under the condition that $\max_{w\in \Leaves{T}} 2\sqrt{\log(2nd/\delta^2)}n^{-1/2}\pi^{-1/2}(w) \leq \log(3/2)$ which is implied by Assumption \ref{assump:C}(i) since $\rho \in (0,1)$.
 Therefore with probability at least $1-\delta^2/n$ we have $$\Delta \leq \bar \Delta$$

Then, conditionally on $\{\Delta \leq \bar\Delta\}$ we have by the perturbation Lemma \ref{lem:perturbV} we have
$$ \Pr{Z^* \in A \mid  X^n_{-\infty}} \leq  \Pr{Z  \in A^{\bar\delta} } + C\bar\delta^{-2}\bar\Delta \log(d) $$

Next we collect the bounds, letting $$r_n := \bar\delta + \{C \sqrt{\log d} + \sqrt{2\log(n/\delta^2)}\} \max_{w\in \Leaves{T}}\varepsilon_n(w).$$ Therefore, by a conditional version of Strassen's theorem, there is a version of $Z$ such that
$$ \wp_n^2 := \Pr{ |\widehat Z^* - Z| > r_n } \leq 2\alpha_\eps + 4\delta^2/n + C\bar\delta^{-2}\bar\Delta \log(d). $$
Then by Markov's inequality, with probability $1-\wp_n$, we have
$$ \Pr{ |\widehat Z^* - Z| > r_n \mid X^n_{-\infty}} \leq \wp_n$$

Then, for some $\vartheta_n \geq \wp_n$, with probability $1-\wp_n$, we have
$$\begin{array}{rl}
\hat {\rm cv}(\delta) \sqrt{1-\epsilon} & \geq_{(1)}  \{{\rm cv}^o(\delta+\wp_n) - r_n \}\sqrt{1-\epsilon}\\
& \geq_{(2)} {\rm cv}^o(\delta+\vartheta_n) - \epsilon {\rm cv}^o(\delta+\vartheta_n)\\
& + \{{\rm cv}^o(\delta+\wp_n)-{\rm cv}^o(\delta+\vartheta_n) - r_n \}\sqrt{1-\epsilon}\\
& \geq_{(3)} {\rm cv}^o(\delta+\vartheta_n) +  \left\{ \frac{\vartheta_n -\wp_n}{C\log d } - r_n \right\}\sqrt{1-\epsilon}-\epsilon{\rm cv}^o(\delta+\vartheta_n)\\
& \geq_{(4)} {\rm cv}^o(\delta+\vartheta_n) \\
\end{array}
$$
where (1) follows by the definition of the quantile function, (2) since $\epsilon \geq 1-\sqrt{1-\epsilon}$, (3) by Theorem \ref{thm:max:anticoncentration} and Assumption \ref{assump:C}(ii) which implies $\bar\sigma \geq c$, and (4) holds provided that
\begin{equation}\label{eq:rel-anticon}   r_n  + \epsilon\frac{{\rm cv}^o(\delta+\vartheta_n)}{\sqrt{1-\epsilon}} \leq \frac{\vartheta_n -\wp_n}{C\log d }\end{equation}
where ${\rm cv}^o(\delta+\vartheta_n) \leq C\sqrt{ \log(d/\delta)}$. Thus (\ref{eq:rel-anticon}) yields
$$\Pr{  \widehat{{\rm cv}}(\delta)\sqrt{1-\epsilon} < {{\rm cv}^o}(\delta+\vartheta_n) } \leq \wp_n $$

To show (\ref{eq:rel-anticon}) we will take $\bar \delta = \left\{ c^{-2}\delta^{-2}\bar\Delta \log^2 n \right\}^{1/2}$ and $\vartheta_n = \delta/\log n$. This implies that $\wp_n^2 \leq 2\alpha_\eps + 4\delta^2/n + c\delta^2/\log^2n$ and that $\vartheta_n -\wp_n \geq \vartheta_n/2$.
Relation (\ref{eq:rel-anticon}) holds provided that
$$ \epsilon \leq  \frac{o(1)\delta}{\log (d) \log^{1/2}(d/\delta)\log (n)},\ \ \ \left\{ c^{-2}\delta^{-2} \frac{\sqrt{\log(nd/\delta^2)}}{n^{1/2}\min_{w\in\Leaves{T}}\pi^{1/2}(w)} \log^2 n \right\}^{1/2} \leq \frac{o(1)\delta}{\log (d) \log n}$$
$$ \alpha_\eps^{1/2} \leq \frac{o(1)\delta}{\log n}, \ \ \sqrt{\log(dn/\delta)}\log(d)\left\{ \gamma + \epsilon + \frac{\sqrt{\log(n|A|/\delta)}}{n^{1/2}\min_{w\in T}\pi^{1/2}(w)}\right\} \leq \frac{o(1)\delta}{\log n} $$
which are implied by Assumption \ref{assump:C}. \\

\underline{Step 3.} Here we show that
$$ \sup_{t \in \mathbb{R}} \left\{ \Pr{ \max_{\substack{w \in T,\\ a\in A}}|M_n(w,a)|  > t } -  \Pr{ \max_{\substack{w \in T,\\ a\in A}}|G_n(w,a)|  > t } \right\}\leq C\delta/\log n $$

We will apply Theorem \ref{thm:clt:max:martingale}. By Proposition \ref{prop:Vbound} we can take $\alpha = \alpha_\eps$, $\Delta_n = 4\gamma|A| + \alpha_\eps$. Then, by Theorem \ref{thm:clt:max:martingale}, for any ${t \in \mathbb{R}}$ we have
$$\Pr{ \max_{\substack{w \in T,\\ a\in A}}|M_n(w,a)|  > t }  \leq \Pr{ \max_{\substack{w \in T,\\ a\in A}}|G_n(w,a)|  > t-C\bar \delta}
 + 2\alpha +  \frac{C\log d}{\bar\delta^2} \Delta_n $$
 $$ +\frac{C\log^2d}{\bar \delta^3} {\displaystyle \sum_{t=1}^n }\Ex{\|\xi_t\|_\infty^3 + \|\eta_t\|_\infty^3}$$
where $G_n = \frac{1}{\sqrt{n}}\sum_{t=1}^n \eta_t$ is a Gaussian process where $\eta_t \sim N(0,\Sigma_t)$, $\Sigma_t = \Exp{t-1}{\xi_t\xi_t'}$, $d=|T|\cdot |A|$.

By Proposition \ref{prop:Sigma} we have that
$$ {\displaystyle \sum_{t=1}^n }\Ex{\|\xi_t\|_\infty^3 + \|\eta_t\|_\infty^3} \leq \frac{C}{n^{1/2}}\sum_{w\in \Leaves{T}} \pi^{-1/2}(w) \leq Cn^{-1/2+\rho}$$

Therefore we have
$$
\begin{array}{rl}
\Pr{ \max_{\substack{w \in T,\\ a\in A}}|M_n(w,a)|  > t }  & \leq  \Pr{ \max_{\substack{w \in T,\\ a\in A}}|G_n(w,a)|  > t-C\bar \delta} \\
 & + 2\alpha_\eps + C\bar\delta^{-2}\{4\gamma|A|+\alpha\}\log (d) + Cn^{-1/2+\rho} \bar\delta^{-3}\log^2(d) \\
  & \leq  \Pr{ \max_{\substack{w \in T,\\ a\in A}}|G_n(w,a)|  > t} \\
 & + 2\alpha_\eps + C\bar\delta^{-2}\{4\gamma|A|+\alpha\}\log (d) + Cn^{-1/2+\rho} \bar\delta^{-3}\log^2(d) \\
 & + C\bar \delta C\log(d)\\
 \end{array}
$$
where we used Theorem \ref{thm:max:anticoncentration}. The result follows under Assumption \ref{assump:C} which implies that for $\bar\delta = o(1)\delta/\{ \log(d) \log n\}$ we have
$$  2\alpha_\eps + C\bar\delta^{-2}\{4\gamma|A|+\alpha_\eps\}\log (d) + Cn^{-1/2+\rho} \bar\delta^{-3}\log^2(d) + C\bar \delta \log(d) \leq \delta/\log n.$$

\end{proof}

\section{Proofs of Auxiliary Lemmas}

\begin{proof} (of Lemma \ref{lem:Sinitial}) We need to bound the largest magnitude of a matrix entry of $\Sum\,(Q-\tilde{Q})\,\Sum'$. This operator acts on $\R^T\otimes \R^A$, and its entries are indexed by pairs $((x,a),(y,b))\in (T\times A)^2$.

One first remark is that such an entry can be nonzero if and only if $x\preceq y$ or vice-versa. To see this, let us consider the effect of applying $\Sum$ to $Q$ and $\tilde{Q}$. By the formulae for the $\Sum$ operator,
\begin{eqnarray*}\Sum\,Q\,\Sum'&=&\sum_{w\in \Leaves{T}}\sum_{x,y\in {\rm Path}_T(w)}\frac{\pi(w)}{\sqrt{\pi(x)\pi(y)}}\,e_xe_y'\otimes Q_w \\ &=&\sum_{x,y\in T\times T}\frac{e_xe_y'}{\sqrt{\pi(x)\pi(y)}}\otimes \sum_{w\in \Leaves{T}:x,y\in {\rm Path}_T(w)}\pi(w)\,Q_w.\end{eqnarray*}
We see at once that in order for the entry $Q((x,a),(y,b))$ to be nonzero one needs that there be some leaf $w$ such that $x,y\in {\rm Path}_T(w)$, in which case $x,y$ are linearly ordered. The same property also holds for $\tilde{Q}$.

We have seen that entries $((x,a),(y,b))$ of $\Sum\,(Q-\tilde{Q})\,\Sum'$ are zero unless $x\preceq y$ or vice-versa. We now wish to bound the magnitude of the nonzero entries (and thus the $\ell_\infty$ norm) of this matrix. So we consider $((x,a),(y,b))$ with $x\preceq y$ (without loss of generality). We need to show that
\[\mbox{\bf Want: } |[\Sum\,(Q-\tilde{Q})\,\Sum']((x,a),(y,b))|\leq \max_{w\in \Leaves{T}:x,y\in {\rm Path}_T(w)}\|Q_w - \tilde{Q}_w\|.\]
To prove this, we must write the LHS from the definition. Note that a leaf $w$ has $x,y\in {\rm Path}_T(w)$ if and only if $w\succeq y$. We deduce that
\begin{eqnarray*}|[\Sum\,(Q-\tilde{Q})\,\Sum']((x,a),(y,b))| &=& \left|\frac{\sum_{w\in \Leaves{T}:x,y\in {\rm Path}_T(w)}\pi(w)\,(Q_w(a,b) - \tilde{Q}_w(a,b))}{\sqrt{\pi(x)\,\pi(y)}}\right|\\ &\leq & \frac{\sum_{w\in \Leaves{T}:x,y\in {\rm Path}_T(w)}\pi(w)}{\sqrt{\pi(x)\,\pi(y)}}\,\max_{w\in \Leaves{T}}\|Q_w - \tilde{Q}_w\|.\end{eqnarray*}
To finish, we will show that
\begin{equation}\label{eq:sumpis}\frac{\sum_{w\in \Leaves{T}:x,y\in {\rm Path}_T(w)}\pi(w)}{\sqrt{\pi(x)\,\pi(y)}}\leq 1.\end{equation}
Recall that $T$ is a complete tree. Therefore, the event that $X^{-1}_{-|y|}=y$ coincides with $T(X^{-1}_{-\infty})\succeq y$, which is the same as saying that $T(X^{-1}_{-\infty})=w$ for some leaf $w\succeq y$. We conclude:
\[\pi(y):= \Pr{X^{-1}_{-|y|}=y} = \sum_{w\in \Leaves{T}\,:\,w\succeq y}\Pr{T(X^{-1}_{-\infty})=w} = \sum_{w\in \Leaves{T}\,:\,w\succeq y}\pi(w).\]
Since $y\in {\rm Path}_T(w)$ is the same as $w\succeq y$,
\[\frac{\sum_{w\in \Leaves{T}:x,y\in {\rm Path}_T(w)}\pi(w)}{\sqrt{\pi(x)\,\pi(y)}} \leq  \sum_{w\in \Leaves{T}\,:\,w\succeq y}\pi(w)  = \sqrt{\frac{\pi(y)}{\pi(x)}}.\]
But also have that $\pi(x)\geq \pi(y)$: $x\preceq y$ implies $\{X^{-1}_{-|x|}=x\}\supset \{X^{-1}_{-|y|}=y\}$. This shows that \eqnref{sumpis} holds and finishes the proof.\end{proof}

\begin{proof} (of Proposition \ref{prop:Sigma}) We have the explicit formulae
\[\Leaves{\xi}_t(w,a) = \left(\Ind{\{T(X^{t-1}_{-\infty})=w\}}\,\frac{e_w}{\sqrt{\pi(w)n}}\right)\otimes (\Ind{\{X_t=a\}} - p(a|X^{t-1}_{-\infty})),\]
and
\[\xi_t(x,a) = \left(\Ind{\{T(X^{t-1}_{-\infty})\succeq x\}}\,\frac{e_x}{\sqrt{\pi(x)n}}\right)\otimes (\Ind{\{X_t=a\}} - p(a|X^{t-1}_{-\infty})).\]
Since $T(X^{t-1}_{-\infty})\succeq x$ implies $\pi(T(X^{t-1}_{-\infty}))\leq  \pi(x)$, we obtain $\|\xi_t\|_{\infty} = \|\Leaves{\xi}_t\|_\infty$ and
\[\Ex{\|\xi_t\|_{\infty}^3} \leq \Ex{\frac{1}{n^{3/2}\,\pi(T(X^{t-1}_{-\infty}))^{3/2}}}= \frac{1}{n^{3/2}}\sum_{w\in \Leaves{T}}\frac{\pi(w)}{\pi(w)^{3/2}} = \frac{1}{n^{3/2}} \sum_{w\in \Leaves{T}}\frac{1}{\pi(w)^{1/2}}.\]
Let us now consider the covariance matrices. Clearly,
\[\Leaves{\Sigma}_t:=\Exp{t-1}{\Leaves{\xi}_{t}{\Leaves{\xi}_{t}}'} = \left(\sum_{w\in \Leaves{T}}\Ind{\{T(X^{t-1}_{-\infty})=w\}}\,\frac{e_we_w'}{\pi(w)\,n}\right)\otimes C_t,\]
where
\[C_{t}(a,a'):=\left\{\begin{array}{ll}p(a|X^{t-1}_{-\infty})(1-p(a|X^{t-1}_{-\infty})),&  a=a';\\ - p(a|X^{t-1}_{-\infty})\,p(a'|X^{t-1}_{-\infty}), & a\neq a'.\end{array}\right.\]
We also have
\[\Sigma_t:=\Exp{t-1}{\xi_{t}\xi_{t}'} = \left(\Sum\,\sum_{w\in \Leaves{T}}\Ind{\{T(X^{t-1}_{-\infty})=w\}}\,\frac{e_we_w'}{\pi(w)\,n}\,\Sum'\right)\otimes C_t,\]
We will need a square root for $\Sigma_t$. Note that $C_t$ has the form:
\begin{equation}\label{eq:defCt} C_t = {\rm diag}(p_t) - p_tp_t'\end{equation}
where $p_t\in\R^A$ is as above. The usual fact that $\Var{X} = \Ex{(X-\Ex{X})^2}$ translates into:
\[x'({\rm diag}(p_t) - p_tp_t')x = (x - (p_t'x)\,{\bf 1})'{\rm diag}(p_t)\,(x - (p_t'x)\,{\bf 1}),\]
that is,
\[C_t = (I-p_t{\bf 1}')\,{\rm diag}(p_t)\,(I-{\bf 1}p_t')\]
 where ${\bf 1}$ is the all-ones vector. This implies that
\[C^{1/2}_t = (I-p_t{\bf 1}')\,{\rm diag}(\sqrt{p_t})\]
is a valid square root for this matrix, and
\[\Sigma^{1/2}_t = \left(\sum_{w\in \Leaves{T}}\Ind{\{T(X^{t-1}_{-\infty})=w\}}\,\frac{\Sum\,e_we_w'}{\sqrt{\pi(w)\,n}}\right)\otimes [(I-p_t{\bf 1}')\,{\rm diag}(\sqrt{p_t})]\]
is a valid square root of $\Sigma_t$.
Finally, consider $\Sigma^{1/2}_t\,N$ where $N=_dN(0,I)$ is independent from $\sF_{t-1}$. Given $x\in T$,
\[ (\Sigma^{1/2}_t\,N)(w,a) = \sum_{w\in \Leaves{T}\,:\, w\succeq x}\frac{\sqrt{p_t(a)}N(w,a) - p_t(a)\sum_{b\in A}\sqrt{p_t(b)}N(w,b)}{\sqrt{\pi(x)n}}\,\Ind{\{T(X^{t-1}_{-\infty})=w\}}.\]
So once again $\|\Sigma^{1/2}_t\,N\|_{\infty}$ is achieved at a leaf, and:
\begin{multline*}\Ex{\|\Sigma^{1/2}_t\,N\|_{\infty}^3\mid \sF_{t-1}}  \\= \sum_{w\in\Leaves{T}}\frac{\Ind{\{T(X^{t-1}_{-\infty})=w\}}}{(\pi(w)n)^{3/2}}\,  \Ex{\max_{a\in A}|\sqrt{p_t(a)}N(w,a) - p_t(a)\sum_{b\in A}\sqrt{p_t(b)}N(w,b)|^3\mid\sF_{t-1}}.\end{multline*}
Now, for each $a$ we have:
\[\Ex{|\sqrt{p_t(a)}N(w,a) - p_t(a)\sum_{b\in A}\sqrt{p_t(b)}N(w,b)|^2\mid\sF_{t-1}} = p_t(a),\]
so
\begin{eqnarray*}& \Ex{\max_{a\in A}|\sqrt{p_t(a)}N(w,a) - p_t(a)\sum_{b\in A}\sqrt{p_t(b)}N(w,b)|^3\mid\sF_{t-1}} \\ \leq & \Ex{\sqrt{\sum_{a\in A}|\sqrt{p_t(a)}N(w,a) - p_t(a)\sum_{b\in A}\sqrt{p_t(b)}N(w,b)|^6}}\leq C.\end{eqnarray*}
We conclude that
\[\Ex{\|\Sigma^{1/2}_t\,N\|_{\infty}^3}\leq \sum_{w\in\Leaves{T}}\frac{C\,\pi(w)}{\{\pi(w)n\}^{3/2}}\leq \frac{C}{n^{3/2}}\sum_{w\in \Leaves{T}}\sqrt{\frac{1}{\pi(w)}}.\]
\end{proof}

\begin{proof} (of Proposition \ref{prop:Vbound}) For the first assertion, by Lemma \ref{lem:Sinitial}, it suffices to show that $\|\Leaves{V}-\Leaves{V}_n\|\infty\leq \epsilon$ and $\Leaves{V}_n\preceq \Leaves{V} +  V_\delta^*$ in the typicality event.  This is what we do below.

In computing $\Leaves{V}_n$, we collect terms that have a given $w$ in them, and obtain:
\[\Leaves{V}_n = \sum_{w\in \Leaves{T}}\frac{N_{n-1}(w)}{\pi(w)\,n}\,e_we_w' \otimes \overline{C}_w,\]
where
\begin{equation}\overline{C}_w:=\left(\frac{1}{N_{n-1}(w)}\sum_{t=h^*}^n\Ind{\{T(X^{t-1}_{-\infty})=w\}}\,C_t\right),\end{equation}
with $C_t$ as in \eqnref{defCt}.

Recall from (\ref{def:Cw}) that $C_w\in \R^{A\times A}$ is given by:
\[C_w(a,a'):= \left\{\begin{array}{ll}p(a|w)(1-p(a|w)),&  a=a';\\ - p(a|w)\,p(a'|w), & a\neq a'.\end{array}\right. .\]
By the continuity assumption, for each index $t$ with $T(X^{t-1}_{-\infty})=w$ we have:
\[\|C_t -  C_w\|_\infty\leq 2\gamma.\]
Moreover, since both $C_t$ and $C_w$ have the form prescribed in Lemma \ref{lem:comparing} below,
\[C_t\preceq C_w + (2\sqrt{|A|}+1)\,\gamma\,I_{A\times A}.\]
Therefore, under continuity, we have that for all $w\in\Leaves{T}$:
\begin{equation}\label{eq:compateCCbar}\|\overline{C}_w -  C_w\|_\infty\leq 2\gamma\mbox{ and }\overline{C}_w\preceq C_w + (2\sqrt{|A|}+1)\,\gamma\,I_{A\times A}.\end{equation}
We now compare $V_n^*$ to the deterministic operator $\Leaves{V}$ introduced in (\ref{eq:defVstar}), which we rewrite below.
\[\Leaves{V} = \sum_{w\in \Leaves{T}}e_we_w' \otimes C_w.\]
We note that
\begin{eqnarray*}\|\Leaves{V}_n - \Leaves{V}\|_\infty &\leq & \left\|\sum_{w\in \Leaves{T}}\left(\frac{N_{n-1}(w)}{\pi(w)\,n}-1\right)\,e_we_w' \otimes \overline{C}_w\right\|_\infty\\ & & + \left\|\sum_{w\in \Leaves{T}}e_we_w' \otimes (C_w-\overline{C}_w)\right\|_\infty\\ \mbox{(typicality + (\ref{eq:compateCCbar}))}& \leq & \epsilon + 2\gamma. \end{eqnarray*}
In addition, under typicality (and noting that each $\overline{C}_w\succeq 0$),
\[\Leaves{V}_n\preceq \sum_{w\in \Leaves{T}}(1+\eps)\,e_we_w' \otimes \overline{C}_w\]
Then (\ref{eq:compateCCbar}) gives:
\[\Leaves{V}_n\preceq \sum_{w\in \Leaves{T}}e_we_w' \otimes [(1+\eps)\,C_w + (1+\eps)(2\sqrt{|A|}+1)\,\gamma \,I_{A\times A}] = \Leaves{V} + \Leaves{V}_\delta\]
as per (\ref{eq:defVdeltastar}).
 \end{proof}

\section{Additional technical lemmas}

\subsection{Comparing covariance matrices}
\begin{lemma}\label{lem:comparing}Assume $p,q\in\R^A$ are two probability vectors. Then:
\[{\rm diag}(p) - pp'\preceq {\rm diag}(q) - qq' + (2\sqrt{|A|}+1)\,\max_{a\in A}|p(a) - q(a)|\,I_{A\times A}.\]\end{lemma}
\begin{proof} Fix $x\in \R^A$. Our goal is to show that:
\[{\bf Goal: } \,x'{\rm diag}(p)\,x - (p'x)^2\leq x'{\rm diag}(q)x - (q'x)^2 + (2\sqrt{|A|}+1)\,\max_{a\in A}|p(a) - q(a)|\,x'x.\]
Fix $x\in \R^A$. Then:
\[x'{\rm diag}(p)\,x = \sum_{a\in A}p(a)x(a)^2\leq \sum_{a\in A}q(a)x(a)^2
 + \max_{a\in A}|q(a)-p(a)|\,x'x.\]
Moreover, since $p+q$ has $\ell^1$ norm bounded by $2$,
\begin{eqnarray*}(p'x)^2 - (q'x)^2 &=& ((p+q)'x)\,((p-q)'x)\\ &\leq &  2\|x\|_\infty\,\|p-q\|_2\,\sqrt{x'x}\\ &\leq &2\sqrt{|A|}\,\max_{a\in A}|p(a) - q(a)|\,x'x.\end{eqnarray*}
Combining these two displays gives us our goal.\end{proof}

\end{document}